\documentclass[12pt,pdftex,noinfoline]{imsart}

\RequirePackage[OT1]{fontenc}
\usepackage{amsthm,amsmath,amsfonts,natbib,mathtools,amssymb}
\RequirePackage{hypernat}
\usepackage[ruled,section]{algorithm}
\usepackage[noend]{algorithmic}
\usepackage{graphicx}
\usepackage{verbatim}
\usepackage{times}
\usepackage[T1]{fontenc}
\usepackage[text={6.0in,8.6in},centering]{geometry}
\usepackage{subfigure}
\usepackage{graphicx}
\usepackage{hyperref}
\usepackage{yhmath}
\usepackage[usenames,dvipsnames]{xcolor}
\definecolor{shadecolor}{gray}{0.9}
\definecolor{shadecolor}{gray}{0.9}
\hypersetup{citecolor=MidnightBlue}
\hypersetup{linkcolor=MidnightBlue}
\hypersetup{urlcolor=MidnightBlue}
\usepackage{enumerate}
\usepackage{fullpage}

\usepackage{tikz}
\usetikzlibrary{arrows,positioning}
\usetikzlibrary{decorations.pathreplacing}

\let\hat\widehat
\let\tilde\widetilde
\newcommand{\Var}{\mbox{Var}\/}
\newcommand{\Sub}{\text{\it Sub}\/}
\newcommand{\reals}{{\mathbb R}}
\newcommand{\E}{{\mathbb E}}

\newcommand{\eps}{{\epsilon}}

\newcommand{\R}{{\mathbb R}}
\newcommand{\LL}{{\mathcal L}}
\newcommand{\II}{{\mathcal I}}
\renewcommand{\P}{{\mathbb P}}

\newcommand{\A}{{\mathcal{A}}}
\newcommand{\M}{{\mathcal{M}}}

\newcommand{\C}{{\mathcal{C}}}

\newcommand{\I}{{\mathcal{I}}}

\newcommand{\T}{{\mathcal{T}}}

\newcommand{\Z}{{\mathbb{Z}}}
\newcommand{\K}{{\mathcal{K}}}

\newcommand{\F}{{\cal F}}

\newcommand{\ham}{{\Upsilon}}

\def\qt#1{\qquad\text{#1}}
\def\argmin{\mathop{\rm argmin}}
\def\argmax{\mathop{\rm argmax}}

\newtheorem{theorem}{Theorem}[section]

\newtheorem{lemma}[theorem]{Lemma}
\newtheorem{corollary}[theorem]{Corollary}
\newtheorem{remark}{Remark}[section]

\numberwithin{equation}{section}

\begin{document}

\begin{frontmatter}
\title{Denoising Flows on Trees}
\runtitle{Denoising Flows on Trees}
\runauthor{Chatterjee and Lafferty}

\begin{aug}
\vskip10pt
\author{\fnms{Sabyasachi} \snm{Chatterjee${}^{*}$}\ead[label=e1]{sabyasachi@galton.uchicago.edu}}
\and
\author{\fnms{John}
  \snm{Lafferty${}^{*\dag}$}\ead[label=e3]{lafferty@galton.uchicago.edu}}
\vskip10pt
\address{${}^*$Department of Statistics\\
${}^\dag$Department of Computer Science\\
University of Chicago
\\[10pt]
\today\\[5pt]
\vskip10pt
}
\end{aug}

\begin{abstract}
We study the estimation of flows on trees, a structured generalization
of isotonic regression.  A tree flow is defined recursively as a
positive flow value into a node that is partitioned into an outgoing
flow to the children nodes, with some amount of the flow possibly
leaking outside.  We study the behavior of the least squares estimator
for flows, and the associated minimax lower bounds. We 
characterize the risk of the least squares estimator in two
regimes. In the first regime the diameter of the tree grows at most
logarithmically with the number of nodes.  In the second regime, the
tree contains many long paths.  The results are compared with known risk
bounds for isotonic regression.  
\end{abstract}
 
\vskip20pt 
\end{frontmatter}

\maketitle

\vskip10pt
\tikzset{circ/.style={circle, draw, fill=white, scale=1}}

\section{Introduction}

We study the problem of denoising tree flows, a graph-structured
generalization of isotonic regression.  In isotonic (monotonic)
regression, a sequence 
$Y_i = \mu_i + \epsilon_i$, for $i=1,\ldots, n$,
is a noisy observation of a monotonic sequence $\mu_1 \geq \mu_2
\geq \cdots \geq \mu_n$, with $\epsilon_i$ independent mean zero
random noise. Tree flows are a generalization of monotonic
sequences. Each node in a (rooted) tree is labeled with a value that can be
thought of as the incoming flow of some fluid. The flow is
partitioned and redirected to the children nodes. 
In a noisy flow, the observation at node $i$ is
$Y_i=\mu_i+\epsilon_i$, where $\mu_i$ is the true flow and
$\epsilon_i$ is mean zero random noise. 

Figure~\ref{fig:flow} illustrates a flow on a tree
with $n=9$ nodes.  The root node has an incoming flow of $6$ units; its left child receives a flow
of 3 and its right child receives a flow of $2$.  The
root node thus ``leaks'' a flow of $1$.  The node having a flow of
1 is seen to leak a flow of $\frac{1}{6}$. In general, the flow
$\mu_i$ at a node $i$ having children $\C(i)$ must satisfy $\mu_i \geq
0$ and the flow constraints
\begin{equation}
\mu_i \geq \sum_{j\in\C(i)} \mu_j.
\end{equation}
More explicitly, suppose that the tree $T_n$ has $n$ nodes;
the leaves of the tree are the nodes $i$ for which $\C(i)$ is empty.
The set of flows $\F(T_n)\subset \reals^n$ is the closed convex cone 
\begin{equation}
\F(T_n) = \Bigl\{\mu\in\reals^n \;:\; \mu_i \geq
\sum_{j\in\C(i)} \mu_j, \;\; \text{for all $i=1,2,\ldots, n$}\Bigr\}.
\label{eq:flowcons}
\end{equation}
By convention, $\mu_1$ will denote the flow to the root node,
and if $\C(i)=\emptyset$, meaning that node $i$ is a leaf, then $\sum_{j\in\C(i)} \mu_j = 0$. Thus
$\mu_i \geq 0$ for all $i=1,2\ldots, n$.

\begin{figure}[t]
\begin{center}
\begin{tikzpicture}
\node (l1) [circ]{6};
\node (l2) [circ, below left=2cm of l1]{3};
\node (l3)[circ, below right=2cm of l1]{2};
\node (l5)[circ, below left=1cm of l2]{$1$};
\node (l4)[circ, below right=1cm of l2]{$2$};
\node (l6)[circ, below left=1cm of l3]{$2$};
\node (l7)[circ, below right=1cm of l3]{$0$};
\node (l8)[circ, below left=1cm of l5, scale=.9]{$\frac{1}{2}$};
\node (l9)[circ, below right=1cm of l5, scale=.9]{$\frac{1}{3}$};
\draw[-] (l1) to node [auto] {} (l2);
\draw[-] (l1) to node [auto] {} (l3);
\draw[-] (l2) to node [auto] {} (l4);
\draw[-] (l2) to node [auto] {} (l5);
\draw[-] (l3) to node [auto] {} (l6);
\draw[-] (l3) to node [auto] {} (l7);
\draw[-] (l5) to node [auto] {} (l8);
\draw[-] (l5) to node [auto] {} (l9);
\end{tikzpicture}
\end{center}
\caption{An example flow. The value $\mu_j$ of a node $j$ is the incoming
  flow. The flow $\mu_j$ at a node
$j$ with two children $l$ and $r$ satisfies $\mu_{j} \geq \mu_{l} +
  \mu_{r}$.  The amount of flow \textit{leaked} at the node is
$\mu_j - \mu_l - \mu_r$. In the above example, the root has flow $6$ 
and leaks $1$ unit of flow. The node having flow $1$ leaks
$\frac{1}{6}$.
Flows generalize monotonic sequences since the flow along each path 
from the root to a leaf must be a nonincreasing sequence.}
\label{fig:flow}
\end{figure}
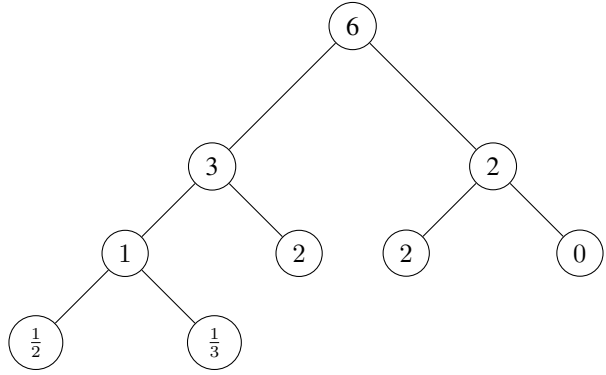

Noisy flows can be seen as arising naturally in certain applications. For
instance, suppose one seeks to estimate
the population of a certain species in a large geographic region. A
sampling survey may give a noisy estimate of the population $\mu_R$ in one
region $R$, and separate estimates might be obtained in a few nonoverlapping
subregions $S_1, S_2,\ldots, S_d$ with $S_i\subset R$ and $S_i \cap
S_j = \emptyset$.  Each subregion could be recursively partitioned.  
The populations are governed by the flow
constraint $\mu_R \geq \mu_{S_1} + \cdots + \mu_{S_d}$.
We refer to \cite{skinner89} for a treatment of classical survey sampling
and its connection to hierarchical modeling.

As another example, computer program profiling 
provides a set of techniques to analyze
the runtime behavior of a program, measuring
the time and storage used in different parts of the code
\citep{gprof,spivey04}.
Profiling software typically generates a tree (or 
a more general directed graph) showing the execution time
spent in different blocks of code. The flow constraints
pertain since the time spent in calls to a specific function must be greater than total
time spent in blocks of code that are reached from that function. 
Profiling is achieved by instrumenting the compiled code with
instructions to monitor its performance. But instrumentation
may change the performance characteristics of the program.
Statistical profilers use sampling to allow the
program to operate closer to its true execution behavior,
with fewer side-effects. More advanced statistical estimation,
such as the structured regression techniques we study
here, could enable more efficient sampling schemes.

Another example arises in hierarchical classification and information
retrieval with respect to a taxonomy of topics, as can occur in
collaborative filtering settings.  For a given search query or
document to be classified, each node in the taxonomy is assigned a
number, representing the relevance of the node's topic to the query.
For instance, the search query might be ESPN and the topic
representing sports might have children nodes basketball, baseball and
football.  If the relevance of the parent (sports) to the query (ESPN)
is required to be no less than the sum of the relevance values to the
children nodes (basketball, baseball and football), this represents a
flow constraint as defined in \eqref{eq:flowcons}. The use of such
constraints is called ``sum-based hierarchical smoothing'' by
\cite{benabbas2011efficient}. Such potential applications
notwithstanding, our main motivation is that 
flow estimation is a natural generalization of 
isotonic regression, which deserves study in its own right.

We consider the problem of estimating a flow $\mu \in \F(T_n)$
for a given tree $T_n$ from noisy observations 
\begin{equation}\label{eq:RegMdl}
Y_{i} = \mu_{i} + \eps_{i}, \qquad \mbox{for } i=1,\ldots, n
\end{equation}
where $\mu \in \F(T_n)$ is unknown and the random errors 
$\epsilon_{i} \sim N(0,\sigma^2)$ are independent with $\sigma^2$
unknown.  A natural estimator for $\mu$ is the least squares
estimator (LSE) $\hat{\mu}$, defined according to
\begin{equation}\label{eq:LSE}
\hat{\mu} := \argmin_{f \in \F(T_n)} \;\; \sum_{i = 1}^{n} (Y_{i} - f_{i})^2.  
\end{equation}
The LSE $\hat \mu$ is uniquely defined, as it is the  projection of $Y$ onto the closed convex set $\F(T_n) \subset \R^{n}$. 
We study the behavior of $\hat{\mu}$ under the squared error loss 
\begin{equation}\label{loss}
\ell(\mu,\mu^{\prime}) := \frac{1}{n} \sum_{i = 1}^{n} (\mu_{i} - \mu^{\prime}_{i})^2.
\end{equation}
The risk of any estimator $\tilde{\mu}$ of $\mu$ under the
squared loss function is given by 
\begin{equation}
R(\mu, \tilde{\mu}) := \E_{\mu} \ell(\tilde{\mu}, \mu),
\end{equation}
where $\E_{\mu}$ denotes the expectation taken with respect to $Y$
having the distribution given by~\eqref{eq:RegMdl}. 
In particular, $R(\mu, \hat{\mu})$ denotes the risk of the LSE.

We study the behavior of the LSE in a setting
where the number of nodes $n$ increases in a sequence of trees $T_n$.
The central statistical questions we investigate include the following.
\begin{itemize}
\item For a given sequence of trees $T_n$, what is the behavior of the
  risk of the least squares estimator $R(\hat{\mu},\mu)$?
Is it consistent, in the sense that $R(\hat{\mu},\mu) \longrightarrow
0$ as the number of nodes $n$ increases? If so, what is the rate of
convergence of the risk of the LSE $R(\mu, \hat{\mu})$? 
How does $R(\hat{\mu},\mu)$ depend on the choice of the sequence of trees?
\vskip10pt
\item What is the fundamental limit of estimation in the minimax
  sense? In other words, what is the scaling as $n\rightarrow\infty$
of the quantity 
\begin{equation}
\Delta_n = \inf_{\tilde{\mu}} \sup_{\mu \in \F(T_n)} R(\hat{\mu},\mu),
\end{equation}
and how does the minimax rate of estimation depend on the structure of
the underlying trees?
\end{itemize}
We provide some answers to these questions in this paper, which
appears to be the first time flow estimation has been studied
from a statistical perspective.

Flow estimation is a generalization of the well studied
isotonic vector estimation problem. In particular, if the tree is the path graph $L_n$,
the problem is to estimate $\mu = (\mu_1, \dots, \mu_n) \in \R^n$ from
observations
\begin{equation*}
  Y_i = \mu_i + \epsilon_i \qt{for $i = 1, \dots, n$}
\end{equation*}
under the constraint $\mu_1 \geq \dots \geq \mu_n \geq 0$. This is, of
course, a special case of univariate isotonic regression and has a
long history; see e.g.,~\cite{Brunk55,AyerEtAl55,vanEeden58}. The risk
of the LSE for isotonic regression has been studied by a number of authors, including
\citet{vdG90,vdG93,Donoho91,BM93,Wang96,MW00,Zhang02,chatterjee2015risk}. It
is shown by \cite{Zhang02} that the risk satisfies
\begin{equation}\label{motw}
R(\mu, \hat{\mu}) \leq C \left\{\left(\frac{\sigma^2
    V(\mu)}{n} \right)^{2/3} +  \frac{\sigma^2 \log n}{n} \right\},
\end{equation}
with $V(\mu) := \mu_1 - \mu_n$, where $C$ is a universal positive constant.
This result shows that the risk of $\hat{\mu}$ scales
as $n^{-2/3}$ provided $V(\mu)$ is
bounded from above by a constant; this is in fact 
the minimax rate of estimation in this problem (see
e.g.,~\citet{Zhang02}).

More broadly, flow denoising is an example
of graph-based signal estimation, a topic
that is of increasing recent interest. To mention a few recent results
in this vein, a lasso-type penalized estimator has
been proposed by~\citet{sharpnack2013graph} to estimate sparse
signals; \cite{wang2014trend} propose adapting trend
filtering ideas from nonparametric
regression \citep{tibshirani2014adaptive,kim2009ell_1}. 
Such approaches are based on penalized empirical risk, and require
a tuning parameter that can be difficult to set in practice.
In contrast, the LSE for flow denoising requires no tuning parameters;
in this way it resembles shape constrained estimation problems
such as convex regression or log-concave density estimation.

The following section presents our approach to investigating the
questions posed above for flow estimation on different families of
trees, and states our technical results. Our main finding is a
surprising gap between the rate of convergence for the least squares
estimator and the minimax rate over all possible estimators, where the
rate for the LSE is not in general monotonic with respect to the depth
of the tree. Section~\ref{sec:proofs} gives the proofs of these
results. Simulations supporting our analysis are provided in
Section~\ref{sec:simulations}.

\section{Results}\label{sec:results}

We analyze two main regimes, corresponding to different assumptions on
the family of trees $T_n$. In the first regime the depth, or diameter, of 
the tree stays bounded or grows at most logarithmically with the number
of nodes $n$. Our second regime of study is a family of trees
containing many long paths. The depth of the trees in this regime
grows at a polynomial rate.

\subsection{Bounded Depth Trees}

In this regime we sharply characterize the risk of the least squares
estimator.  Our first theorem gives an upper bound on the risk
$R(\hat{\mu},\mu)$ in terms of the tree depth (or height) $h_n$,
defined as the maximum graph distance from the root to a leaf. We
denote the Euclidean norm by $\|\cdot\|$, and let $\mu_1$ denote the
flow at the root for any $\mu\in\F(T_n)$.

\begin{theorem}\label{lse_upbd_shallow}
For any tree $T$ with height $h$,
the worst case risk of the the least squares estimator satisfies
\begin{equation}
\sup_{\mu \in \F(T): \mu_1 \leq V} \E \frac{1}{n}\|\hat{\mu} - \mu\|^2 \leq C \frac{1}{n}\Bigl(\sigma^2 h^2 \log \bigl(\frac{n}{h}\bigr) + h V \sigma \sqrt{\log\bigl(\frac{n}{h}\bigr)}\Bigr)
\end{equation}
for some universal constant $C$.
\end{theorem}

This risk bound holds generally, but it is 
mainly useful when the depth of the tree grows very slowly with $n$. For
example, for trees with depth growing logarithmically with
$n$, it implies an almost parametric rate of convergence
$\tilde{O}(1/n)$.  Examples include the complete binary tree and the star graph with $n$ vertices. 

Restricting to the case of bounded depth trees, with height $h_n \leq C$ where $C$ is a
constant, Theorem~\ref{lse_upbd_shallow} shows that the risk of the
least squares estimator scales according to $\sigma^2 {\log n}/{n} + V \sigma
{\sqrt{\log n}}/{n}$. Our next result is a minimax bound
in this setting.

\begin{theorem}\label{simplexminimax}
Let $\F_{n,V} \subset \R^n$ be the space of flows on trees of bounded
depth $h_n \leq C$, with root flow $\mu_1\leq V$. Then
for any $\epsilon > 0$, 
\begin{equation*}
\inf_{\tilde{\mu}} \sup_{\mu \in \F_{n,V}} \E \frac{1}{n} \|\tilde{\mu} - \mu\|^2 \geq 
\begin{cases}
\displaystyle c \frac{V^2}{n} & 0 < V \leq \sigma \sqrt{\log n}\\
\displaystyle c_{\epsilon} V \sigma \frac{\sqrt{\log n}}{n} & \sigma \sqrt{\log n} < V \leq \sigma n^{1 - \epsilon}\\
\displaystyle c \sigma^2 & V > n \sigma,
\end{cases} 
\end{equation*}
where $c$ is a universal constant and
$c_{\epsilon}$ is a constant that depends only on $\epsilon$.
\end{theorem}

\vskip15pt
\begin{remark}
The trivial estimator $\hat\mu = Y$ incurs a
risk of $\sigma^2$, and becomes minimax rate optimal as $V$
increases. 
One expects that the minimax rate grows continuously with $V$, keeping
$n$ and $\sigma$ fixed.  The three cases above are a result of our
proof technique, which only allows $V$ to grow 
as $\sigma n^{1 - \epsilon}$ for some arbitrarily small, but fixed,
$\epsilon$. 
\end{remark}

In the bounded depth case, Theorem~\ref{lse_upbd_shallow} and Theorem~\ref{simplexminimax} 
together show that when $\sigma \sqrt{\log n} < V$, the minimax rate of estimation
is indeed $V \sigma \sqrt{\log n}/n$, and that this is
achieved by the LSE.
However, when $V \leq \sigma \sqrt{\log n}$,
the lower bound given in Theorem~\ref{simplexminimax} is $V^2/n$ which
is attained by the trivial estimator $\hat\mu = 0$.  
If the upper bound in Theorem~\ref{lse_upbd_shallow} is tight, it
indicates that the LSE may not be minimax rate optimal in the small
$V$ regime. The following result shows that this
is indeed the case.

\begin{theorem}\label{lse_lb_shallow}
Under the same assumptions as in Theorem~\ref{lse_upbd_shallow}, 
the worst case risk of the LSE satisfies the lower bound
\begin{equation*}
\sup_{\mu \in \F(T_n): \mu_1 \leq V} \E \frac{1}{n} \|\hat{\mu} - \mu\|^2 \geq C \frac{\sigma^2}{h} \frac{\log({n}/{h})}{n}.
\end{equation*}
\end{theorem}

\vskip15pt
\begin{remark}
For trees with bounded depth, the lower bound
is $C \sigma^2 \log(n)/n$. Theorem~\ref{lse_lb_shallow} actually holds for the zero flow $\mu=0$.
In other words, the risk of the LSE at the origin is
$\sigma^2 \log(n)/n$, up to a constant factor. 
This follows from an argument involving the statistical dimension of
the cone of flows for the star graph; see~\cite{LivEdge}. 
\label{rem:edge}
\end{remark}

\begin{remark}
A simple example of a sequence of trees with bounded depth
is the collection of star graphs, with one root and $n - 1$ children. 
With known root value $\mu_1$, the simplex constraint makes the
estimation problem similar to that of estimating a vector lying
in an $\ell_1$ ball of known radius.
\end{remark}



The above theorems quantify how the flow estimation problem
is ``easier'' in the bounded depth regime than in the setting
of a path.  Specifically, the risk of the LSE in the bounded depth tree setting scales according to
the (nearly) parametric rate $O(\log(n)/n)$, while in the path setting it 
scales as $O(n^{-2/3})$. A natural question is then whether 
paths are the ``hardest'' cases for flow estimation among all rooted
trees, that is, whether $O(n^{-2/3})$ is the slowest rate among all
sequences of trees.  Another question is whether or not the 
rate of convergence of the LSE is monotonic with respect to the depth.
These questions motivate our study of a family of ``deep trees,'' as described in the next section.

\subsection{Deep Trees}
\def\nat{{\text{nat}}}

We define a family of trees $\T_{\alpha,n}$ parameterized by 
$\alpha$ satisfying $0 \leq \alpha \leq 1$.  For a given $n$ and $\alpha$,
the root has $m = \lceil n^{\alpha}\rceil$ children, and each of
these $m$ children is the starting point of a path
of length $l = \lceil n^{1 - \alpha}\rceil$. See Figure~2.
When  $\alpha = 0$ the tree $\T_{\alpha,n}$ is
a single path and flow denoising on $\T_{\alpha,n}$
corresponds to the isotonic sequence regression problem. When $\alpha
= 1$ the tree $\T_{\alpha,n}$ is a star graph with $n$ children of the
root. Hence, this family of trees interpolates between the path 
and star graphs. If the risk of the 
LSE is indeed monotonic with the depth of the underlying sequence of
trees, then it should decay faster with $n$ as $\alpha$
increases. 

\begin{figure}[t]
\begin{center}
\vskip-.1in
\tikzset{circ/.style={circle, draw, fill=black, scale=.5}}
\begin{tabular}{cc}
\hskip-15pt
\begin{tikzpicture}
\node (l1) at (0,0) [circ]{};

\node (l2) at (-2,-2) [circ]{};
\node (l3) at (-.5,-2) [circ]{};
\node (l4) at (2,-2) [circ]{};

\node (l5) at (-2,-3.5) [circ]{};
\node (l6) at (-.5,-3.5) [circ]{};
\node (l7) at (2,-3.5) [circ]{};

\node at (0.3,-2) {$\cdots$};

\node at (0.3,-3.5) {$\cdots$};

\node (d5) at (-2,-4.0) {};
\node (d6) at (-.5,-4.0) {};
\node (d7) at (2,-4.0) {};

\draw[-] (l1) to node [auto] {} (l2);
\draw[-] (l1) to node [auto] {} (l3);
\draw[-] (l1) to node [auto] {} (l4);

\draw[-] (l2) to node [auto] {} (l5);
\draw[-] (l3) to node [auto] {} (l6);
\draw[-] (l4) to node [auto] {} (l7);

\draw[-] (l5) to node [auto] {} (d5);
\draw[-] (l6) to node [auto] {} (d6);
\draw[-] (l7) to node [auto] {} (d7);

\node at (-2,-4.3) {$\vdots$};
\node at (-.5,-4.3) {$\vdots$};
\node at (2,-4.3) {$\vdots$};

\node (l8) at (-2,-5.5) [circ]{};
\node (l9) at (-.5,-5.5) [circ]{};
\node (l10) at (2,-5.5) [circ]{};

\node (l11) at (-2,-7.0) [circ]{};
\node (l12) at (-.5,-7.0) [circ]{};
\node (l13) at (2,-7.0) [circ]{};

\node (d8) at (-2,-5.0) {};
\node (d9) at (-.5,-5.0) {};
\node (d10) at (2,-5.0) {};

\draw[-] (l8) to node [auto] {} (d8);
\draw[-] (l9) to node [auto] {} (d9);
\draw[-] (l10) to node [auto] {} (d10);

\node at (0.3,-5.5) {$\cdots$};

\draw[-] (l8) to node [auto] {} (l11);
\draw[-] (l9) to node [auto] {} (l12);
\draw[-] (l10) to node [auto] {} (l13);

\node at (0.3,-7.0) {$\cdots$};

\draw [decorate,decoration={brace,amplitude=10pt,raise=4pt},yshift=0pt]
      (2.1,-2) -- (2.1,-7) node [black,midway,xshift=1.4cm]
      {\footnotesize $l=\lceil n^{1-\alpha}\rceil$};

\draw [decorate,decoration={brace,mirror,amplitude=10pt,raise=0pt},yshift=0pt]
      (-2,-7.3) -- (2.0,-7.3) node [black,midway,xshift=-.0cm,yshift=-18pt]
      {\footnotesize $m = \lceil n^{\alpha}\rceil$};

\end{tikzpicture}
&
\hskip20pt
\begin{tikzpicture}
\node (l1) at (0,0) [circ]{};

\node (l2) at (-2,-2) [circ]{};
\node (l3) at (-.5,-2) [circ]{};
\node (l4) at (2,-2) [circ]{};

\node (l5) at (-2,-3.5) [circ]{};
\node (l6) at (-.5,-3.5) [circ]{};
\node (l7) at (2,-3.5) [circ]{};

\node at (-.5,0) {$\mu_1$};

\node at (-2.5,-2) {$\mu_1^{(1)}$};
\node at (-1.0,-2) {$\mu_1^{(2)}$};
\node at (0.3,-2) {$\cdots$};
\node at (1.5,-2) {$\mu_1^{(m)}$};

\node at (-2.5,-3.5) {$\mu_2^{(1)}$};
\node at (-1.0,-3.5) {$\mu_2^{(2)}$};
\node at (0.3,-3.5) {$\cdots$};
\node at (1.5,-3.5) {$\mu_2^{(m)}$};

\node (d5) at (-2,-4.0) {};
\node (d6) at (-.5,-4.0) {};
\node (d7) at (2,-4.0) {};

\draw[-] (l1) to node [auto] {} (l2);
\draw[-] (l1) to node [auto] {} (l3);
\draw[-] (l1) to node [auto] {} (l4);

\draw[-] (l2) to node [auto] {} (l5);
\draw[-] (l3) to node [auto] {} (l6);
\draw[-] (l4) to node [auto] {} (l7);

\draw[-] (l5) to node [auto] {} (d5);
\draw[-] (l6) to node [auto] {} (d6);
\draw[-] (l7) to node [auto] {} (d7);

\node at (-2,-4.3) {$\vdots$};
\node at (-.5,-4.3) {$\vdots$};
\node at (2,-4.3) {$\vdots$};

\node (l8) at (-2,-5.5) [circ]{};
\node (l9) at (-.5,-5.5) [circ]{};
\node (l10) at (2,-5.5) [circ]{};

\node (l11) at (-2,-7.0) [circ]{};
\node (l12) at (-.5,-7.0) [circ]{};
\node (l13) at (2,-7.0) [circ]{};

\node (d8) at (-2,-5.0) {};
\node (d9) at (-.5,-5.0) {};
\node (d10) at (2,-5.0) {};

\draw[-] (l8) to node [auto] {} (d8);
\draw[-] (l9) to node [auto] {} (d9);
\draw[-] (l10) to node [auto] {} (d10);

\node at (-2.5,-5.5) {$\mu_{l-1}^{(1)}$};
\node at (-1.0,-5.5) {$\mu_{l-1}^{(2)}$};
\node at (0.3,-5.5) {$\cdots$};
\node at (1.5,-5.5) {$\mu_{l-1}^{(m)}$};

\draw[-] (l8) to node [auto] {} (l11);
\draw[-] (l9) to node [auto] {} (l12);
\draw[-] (l10) to node [auto] {} (l13);

\node at (-2.5,-7.0) {$\mu_{l}^{(1)}$};
\node at (-1.0,-7.0) {$\mu_{l}^{(2)}$};
\node at (0.3,-7.0) {$\cdots$};
\node at (1.5,-7.0) {$\mu_{l}^{(m)}$};
\node at (0, -8.0) {};
\end{tikzpicture}
\end{tabular}
\end{center}
\caption{Left: A family of trees ${\T}_{\alpha, n}$, consisting of
  root with $m=\lceil n^\alpha\rceil $
children, and $m$ paths of length $l=\lceil n^{1-\alpha}\rceil$. Right: A
parameterization of flows in the family $\mathcal{F}_{\alpha, n}
= \mathcal{F}({\T}_{\alpha, n})$. }
\label{fig:kite}
\end{figure}
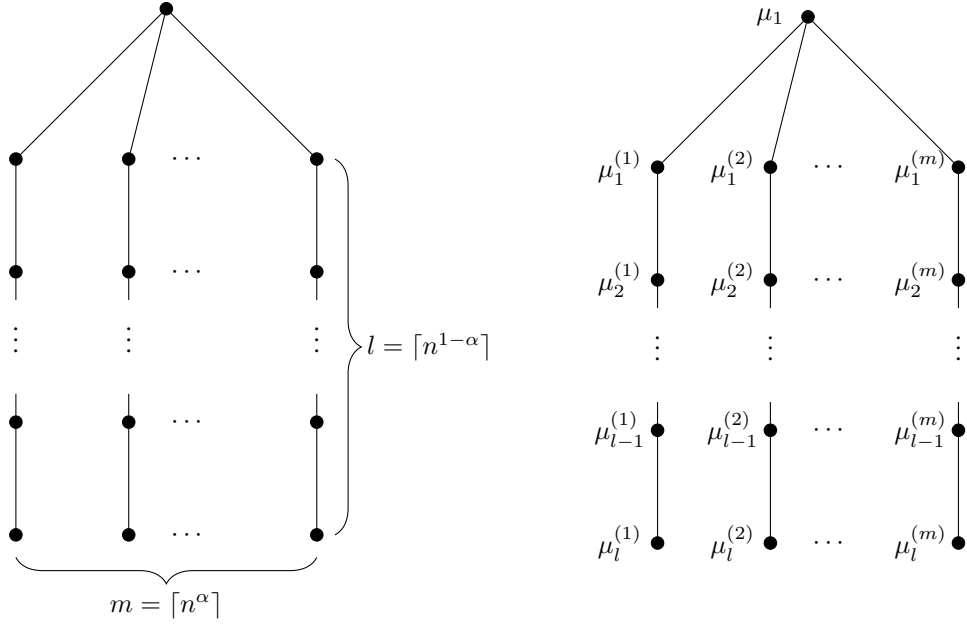


\subsubsection{Upper Bounds for the LSE}

When $\alpha$ is closer to zero than to one, a natural
estimator is to set $\hat{\mu}_1 = Y_1$ at the root and perform
$n^{\alpha}$ separate isotonic regressions on each of the
paths. Since this estimator need not satisfy the flow constraint at
the root, we then project the resulting estimator onto the space
$\F(\T_{n,\alpha})$ to obtain the estimator $\hat{\mu}_\nat$.

Using the results of \cite{Zhang02} for isotonic regression,
together with an application  of H\"older's inequality (see Section~\ref{sec:proofs}),
we can show that the risk of the natural estimator
satisfies 
\begin{equation}
\sup_{\mu \in \F(\T_{\alpha,n}): \mu_1 \leq V} \E \frac{1}{n}
\|\hat{\mu}_\nat - \mu\|^2 \leq \Bigl(\frac{\sigma^2 V}{n}\Bigr)^{2/3} +
\frac{\sigma^2 \log en}{n^{1 - \alpha}}.
\label{natest}
\end{equation}

For small $\alpha$, it is natural to expect an ``isotonic effect''
where the performance of the actual LSE $\hat{\mu}$ should be close to
the performance of the 
natural estimator estimator $\hat{\mu}_\nat$. This intuition is made
rigorous by our first upper bound for the LSE over $\F(\T_{\alpha, n})$.

\begin{theorem}
\label{isotonicupbd}
The risk of the LSE satisfies 
\begin{equation*}
\sup_{\mu \in \F(\T_{\alpha,n}): \mu_1 \leq V} \E \frac{1}{n} \|\hat{\mu} - \mu\|^2 \leq 2\Bigl(\frac{\sigma^2 V \log(en)}{n}\Bigr)^{2/3} + \frac{\sigma^2 \log(en)}{n^{1 - \alpha}}.
\end{equation*}
\end{theorem}


\vskip5pt The proof of Theorem~\ref{isotonicupbd} exploits
monotonicity of the flow along the paths, but does not use the simplex
constraints at every level. The proof technique is based on a general
theory of analyzing least squares estimators which uses an appropriate
notion of size of the tangent cone at flows that are piecewise
constant along every path. This point will be made clear in
Section~\ref{sec:proofs}, where Theorem~\ref{isotonicupbd} is proved.

As $\alpha$ increases, one expects that imposing the
simplex constraint at the root becomes more important.
Since Theorem~\ref{isotonicupbd} is proved only using 
monotonicity along each path, a separate risk bound is needed in the
large $\alpha$ regime. 

This becomes clear if one thinks about the zero estimator $\hat\mu = 0$.
The sum of the squared errors of this estimator is
$V^2$ at every level of the tree; since the
tree has depth $n^{1 - \alpha}$, the sum of squared errors of the zero
estimator is at most $n^{1 - \alpha} V^2$. 
The resulting bound can be better than the upper bound given in
Theorem~\ref{isotonicupbd}. 

As a type of oracle estimator that imposes the flow constraints, 
suppose we know the value of $\mu_1 = V$.  Then we could perform
``simplex regression'' on each level. Specifically, using the notation of
Figure~2, we first compute, for each level $\ell$, the estimator 
\begin{equation}
\tilde \mu_{\ell} = \argmin_{\sum_{j=1}^{m} \mu^{(j)}_\ell \leq
  V} \;\sum_{j=1}^{m} \left(Y^{(j)}_{\ell} - \mu^{(j)}_\ell\right)^2.
\end{equation}
We then project $\tilde\mu$ onto the set of flows.
Let us call this estimator $\hat{\mu}_{oracle}$, since it assumes
knowledge of the root flow $V$. We can then establish the risk bound
\begin{equation}
\E \frac{1}{n} \|\hat{\mu}_{oracle} - \mu\|^2 \leq C\left(V \sigma \frac{\sqrt{\log n}}{n^{\alpha}} + \sigma^2 \frac{\log n}{n^{\alpha}}\right)
\end{equation}
using Theorem~\ref{lse_upbd_shallow}.
For large $\alpha$, one expects that the performance of the LSE
will be similar to this oracle estimator.  Our next result establishes
that this is indeed the case.

\vskip10pt
\begin{theorem}\label{simplexupbd}
The risk of the LSE satisfies 
\begin{equation*}
\E \frac{1}{n} \|\hat{\mu} - \mu\|^2 \leq C \big(\frac{\sigma^2 (\log(en))^3} {n^{\alpha}} + V \frac{\sigma (\log(en))^{3/2}}{n^{\alpha}}\big)
\end{equation*}
for a universal constant $C$.
\end{theorem}

\vskip10pt
Theorem~\ref{simplexupbd} is proved using techniques that control the maxima of a
suitable empirical process, requiring the estimation of covering numbers
for the space of flows. The full proof is given in Section~\ref{sec:simplexupbd}.
We remark that the extra log factor in the bound in
Theorem~\ref{simplexupbd} appears
to be an artifact of our current proof technique,
and may not be necessary.

Theorems~\ref{isotonicupbd} and \ref{simplexupbd} hold for any $0 \leq
\alpha \leq 1$, but are useful in different regimes.  The following
corollary combines them to establish a single risk bound. This is done
by treating $V$ and $\sigma$ as fixed, and considering the
risk as a function of $n$ only. 
The scaling of $n$ is better in Theorem~\ref{isotonicupbd} for $0 \leq
\alpha \leq 1/3$ and is better in Theorem~\ref{simplexupbd} when $1/3
< \alpha \leq 1$. 
\vskip10pt

\begin{corollary}\label{cor1}
The worst case risk of the least squares estimator 
satisfies
\begin{equation*}
\sup_{\mu \in \F(\T_{n,\alpha}): \mu_{1} \leq V} \E \frac{1}{n}
\|\hat{\mu} - \mu\|^2 \leq 
\begin{cases}
\displaystyle\;    2\big(\frac{\sigma^2 V \log(en)}{n}\big)^{2/3} + \frac{\sigma^2 \log(en)}{n^{1 - \alpha}}           & 0 < \alpha \leq \frac{1}{3}\\[10pt]
\displaystyle\;                 \frac{C}{n} \big(\sigma^2 (\log(en))^3 n^{1 - \alpha} + V \sigma (\log(en))^{3/2}\big)              & \frac{1}{3} \leq \alpha < 1.
\end{cases}
\end{equation*}
\end{corollary}


Note that for fixed $V$ and $\sigma$, the upper bound of Corollary~\ref{cor1} is
not monotonically decreasing in $\alpha$.
To resolve whether this is intrinsic to the flow estimation problem,
lower bounds are required, as discussed in the next section.

\subsubsection{Minimax Lower Bounds}

Lower bounds for flow estimation in the family of trees
$\T_{n,\alpha}$ should match the lower bound for isotonic regression
when $\alpha = 0$.  We therefore first establish a minimax lower bound
for isotonic regression; although existing lower bounds are comparable
\citep{chatterjee2015risk}, we are not aware of this exact form of the
lower bound appearing in the literature.

\begin{theorem}\label{isominimax}
The minimax lower bound
\begin{equation*}
\inf_{\tilde{\mu}} \sup_{\mu \in \M_{V}} \E \frac{1}{n} \|\tilde{\mu} - \mu\|^2 \geq \frac{1}{192} \min\Bigl\{\sigma^2,\,V^2,\,\Big(\frac{\sigma^2 V}{n}\Big)^{2/3}\Bigr\}
\end{equation*}
holds for the parameter space $$\M_{V} = \{\mu \in \R^n: V \geq \mu_1 \geq
\dots \geq \mu_n \geq 0\}.$$ 
\end{theorem}

This bound is a minimum of three terms. If $V$ is
not too small or large, the rate $\big(\sigma^2{V}/{n}\big)^{2/3}$ is standard and is attained by the
LSE. If $V$ is very small, the 
$V^2$ term is the smallest, and this rate is achieved by the trivial estimator $\hat\mu=0$.
In this regime, the LSE is not minimax rate optimal since 
it suffers a risk of at least $\sigma^2 {\log (en)}/{n}$ at the
origin; see Remark~\ref{rem:edge}. In case $V$ is very large, the
rate becomes $\sigma^2$, which is achieved by the trivial estimator
$\hat\mu = Y$.

For $\alpha$ small, the hardness of flow estimation on $\T_{n,\alpha}$ 
is dominated by the monotonicity constraints in each path rather than
by the simplex constraints. This leads us to derive a minimax lower bound using Theorem~\ref{isominimax}.

\begin{theorem}\label{minimaxisomain}
Let $\F_{n,\alpha,V}$ denote the set of flows on tree $\T_{n,\alpha}$
with root value at most $V$, and set $m = \lceil n^{\alpha}\rceil$. Then for any $\alpha$,
\begin{equation*}
\inf_{\tilde{\mu}} \sup_{\mu \in \F_{n,\alpha,V}} \E \frac{1}{n} \|\tilde{\mu} - \mu\|^2
\geq  \sup_{(v_1,\dots,v_{m}):\, \sum_{i = 1}^{m} v_i \leq V} \:\: \frac{c}{n} \sum_{i = 1}^{m} \min\Bigl\{n^{1 - \alpha} \sigma^2, \,(v_i \sigma^2)^{2/3} n^{(1 - \alpha)/3}, \,n^{1 - \alpha} v_i^2\Bigr\}. 
\end{equation*}
\end{theorem}

This result is proved by considering the subset of flows
where the $m$ the children of the root have flows $(v_1,\dots,v_{m})$,
together satisfying $\sum_{i = 1}^{m} v_i \leq V$. Estimation in this
parameter space is then equivalent to $m$ separate isotonic regressions and
the minimax lower bound is obtained by using
Theorem~\ref{isominimax}. We obtain, as a corollary, a minimax lower
bound in the regime where $\alpha$ is small and $V$ is neither too small nor
too large for the trivial estimators to dominate.

\begin{corollary}\label{isominimaxcoro}
Let $0 \leq \alpha \leq 1/3$, and suppose that $n^{3 \alpha - 1} \leq V \leq n$. Then 
\begin{equation*}
\inf_{\tilde{\mu}} \sup_{\mu \in \F_{n,\alpha,V}} \E \frac{1}{n} \|\tilde{\mu} - \mu\|^2 \geq c \Bigl(\frac{V \sigma^2}{n}\Bigr)^{2/3}. 
\end{equation*}
\end{corollary}


Thus, under the assumptions of this corollary,
the LSE is minimax rate optimal up to a constant factor.

For the range $1/3 \leq \alpha \leq 1$, we prove lower bounds in a
different manner, using a similar strategy as was used in 
the lower bound for the star graph.

\begin{theorem}\label{simpminimax}
For fixed $1/3 < \alpha \leq 1$,
\begin{equation*}
\inf_{\tilde{\mu}} \sup_{\mu \in \F_{n,\alpha,V}} \E \frac{1}{n} \|\tilde{\mu} - \mu\|^2 \geq 
\begin{cases}
\displaystyle c \frac{V^2}{n^{\alpha}} & \;\;\text{in case}\ \displaystyle 0 < V \leq \frac{\sigma \sqrt{\log n}}{n^{(1 - \alpha)/2}}\\[10pt]
\displaystyle c_{\epsilon} V \sigma \frac{\sqrt{\log n}}{n^{(1 +
    \alpha)/2}} & 
 \;\; \text{for}\;\ \displaystyle \frac{\sigma \sqrt{\log n}}{n^{(1 - \alpha)/2}} < V \leq n^{(1 + \alpha)/2 \:-\: \epsilon} \sigma,\\
\end{cases} 
\end{equation*}
for any $\epsilon  > 0$, where $c$ is a universal constant and
$c_{\epsilon}$ is a constant only depending on $\epsilon$.
\end{theorem}

\vskip10pt
In the first case of this result, the trivial estimator $\hat\mu=0$ attains the
optimal rate. However, in the second case, where $V$ is large, the
lower bound does not match our bound for the LSE in Theorem~\ref{simplexupbd}.
Our next result shows that, in fact, our minimax lower bound is
tight (up to log factors) in this regime.
\vskip10pt

\begin{theorem}\label{minimaxupbd}
For any $1/3 \leq \alpha \leq 1$, we have the upper bound
\begin{equation}
\inf_{\tilde{\mu}} \sup_{\mu \in \F_{n,\alpha,V}} \E \frac{1}{n} \|\tilde{\mu} - \mu\|^2 \leq 21 V\sigma\:\frac{\sqrt{\log(en)}}{n^{(1 + \alpha)/2}} + 16 \sigma^2 \frac{1 + \log(en)}{n}.
\end{equation}
\end{theorem}
 
\vskip5pt
This bound is established by considering a least squares estimator on
an appropriate finite net over the space of flows $\F(\T_{n,\alpha})$. The proof is
thus an information theoretic argument, and does not exhibit
a computationally efficient estimator.

\begin{remark}
As a possible avenue for constructing
a computationally efficient estimator, 
suppose that a procedure can detect which of the children of the root
have small flow value.
If the entire path under such nodes is estimated as zero, and 
otherwise an isotonic regression is fit, this might attain the minimax
rate.  We leave this as a topic for future research.
\end{remark}

\subsection{Lower Bound for the LSE}

Finally, we address the question of whether the LSE is rate optimal
for the range $1/3 < \alpha \leq 1$, noting that the $\tilde O(n^{1 -
  \alpha})$ upper bound of Theorem~\ref{simplexupbd} does
not match the $\tilde \Omega(n^{(1 - \alpha)/2})$ lower bound of 
Theorem~\ref{simpminimax}.  The following result 
asserts that there is an actual gap.

\begin{theorem}\label{lselbd}
For any $V  > 0$, the worst case risk of the least squares estimator
satisfies the lower bound
\begin{equation*}
\sup_{\mu \in \F_{n,\alpha,V}} \E \frac{1}{n} \|\hat{\mu} - \mu\|^2
\geq \begin{cases} 
\pi^2 \sigma^2 n^{\alpha-1} & 1/3 \leq \alpha \leq
  1/2 \\[7pt] 
\pi^2 \sigma^2 n^{-\alpha} & 1/2 \leq \alpha \leq 1.
\end{cases}
\end{equation*}
\end{theorem}

The lower bound in Theorem~\ref{lselbd} actually holds when $\mu=0$ is
the zero flow; the proof of this result is given in Section~\ref{zeroproof}.

Taken together, all of these bounds characterize the exponents
in the worst case risk for the LSE and any possible flow
estimator, as portrayed graphically in Figure~3.
They imply the perhaps surprising conclusion
that the LSE is not rate optimal over the range $1/3\leq \alpha \leq
1$, considering $V$ and $\sigma$ to be fixed constants.

\begin{figure}[t]
\begin{center}
\begin{tabular}{cc}
\\[.75in]
\hskip-.3in
\rm risk exponent & 
\\[-1.6in]
& \hskip-15pt \includegraphics[width=.7\textwidth]{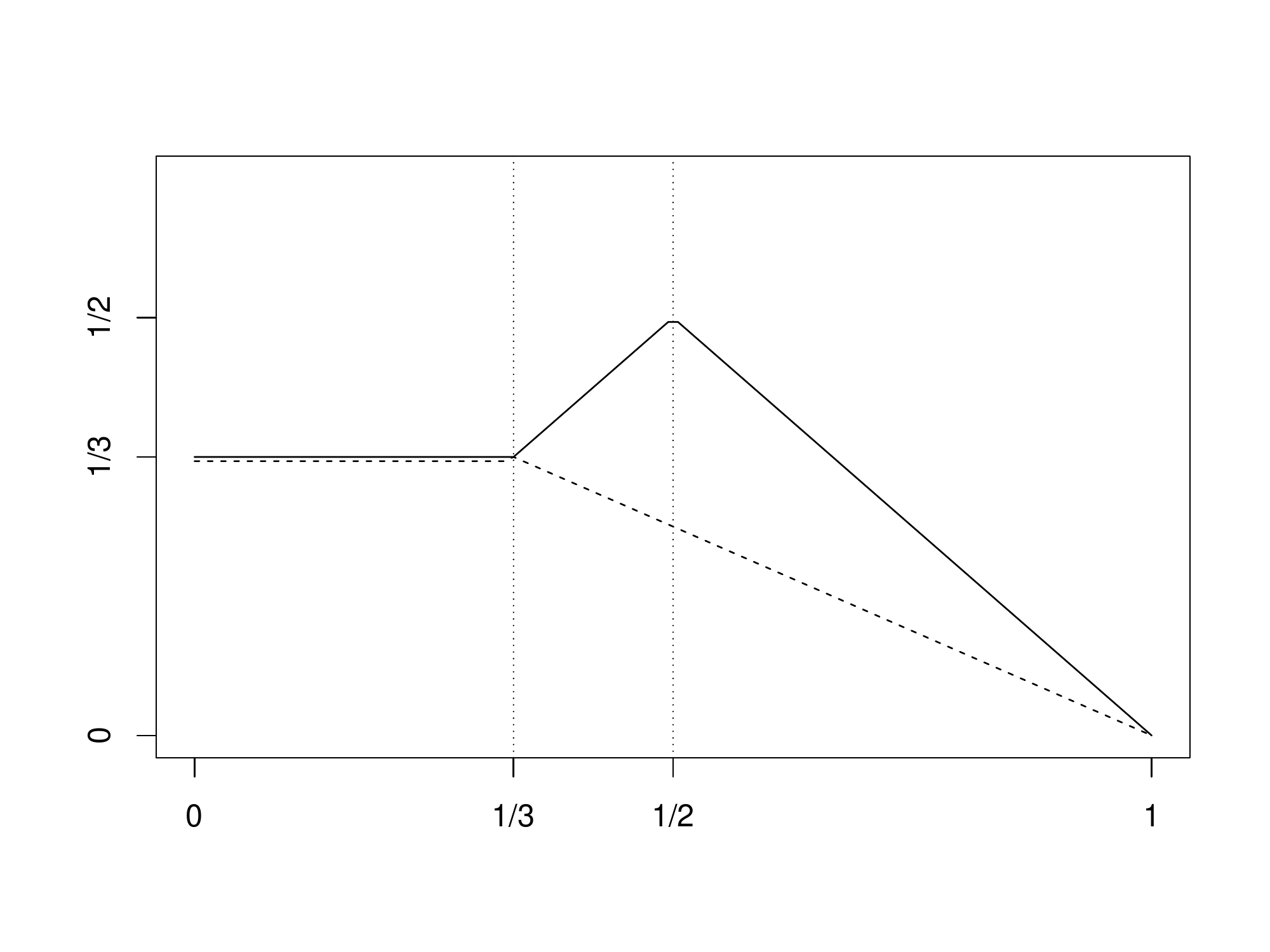}\\[-25pt]
& path length parameter $\alpha$
\end{tabular}
\end{center}
\caption{
With $\alpha=0$ the tree $T_{n,\alpha}$ is a single path,
and with $\alpha=1$ the graph is a star; in general
the tree has $m=\lceil n^\alpha\rceil$ paths, each of length $l=\lceil
n^{1-\alpha}\rceil$.
This plot shows the leading exponent (thus disregarding log factors) in 
the expected squared error $\E \|\hat \mu - \mu\|^2$ (which is $n$
times the risk) of the LSE (solid line) the minimax lower bound (dashed line).
Our sequence of results establishes that these exponents are tight, 
and thus exhibit an unusual gap in a shape-constrained estimation
problem between the performance of the least
squares estimator and the best possible estimator.}
\label{fig:knife}
\end{figure}

\section{Proofs}\label{sec:proofs}

\begin{table}
\begin{center}
\begin{small}
\renewcommand{\arraystretch}{1.3}
\begin{tabular}{|l|l|l|l|}
\multicolumn{2}{c}{\textit{Result}} & \multicolumn{1}{c}{\textit{Proof technique}}  \\
\hline
Theorem~\ref{lse_upbd_shallow} & Upper bound for LSE; shallow trees & Gaussian supremum functional\\
Theorem~\ref{simplexminimax} & Minimax lower bound; shallow trees & Fano's lemma \\
Theorem~\ref{lse_lb_shallow} &  Lower bound for LSE; shallow trees & Gaussian widths \\
Theorem~\ref{isotonicupbd} & Isotonic upper bound for LSE; deep trees & Statistical dimension\\
Theorem~\ref{simplexupbd}&  Simplex upper bound for LSE; deep trees & Chaining and entropy bounds \\
Theorem~\ref{isominimax}&  Minimax lower bound; monotone sequences & Assouad's lemma\\
Theorem~\ref{minimaxisomain}& Minimax lower bound; deep trees, $\alpha \leq \frac{1}{3}$ & Minimax for isotonic regression\\
Theorem~\ref{simpminimax}&  Minimax lower bound; deep trees, $\alpha \geq \frac{1}{3}$ & Fano's lemma\\
Theorem~\ref{minimaxupbd}& Tightness of minimax lower bound, $\alpha \geq \frac{1}{3}$ & Covering numbers, LSE on a net\\
Theorem~\ref{lselbd}& Tightness of LSE upper bound, $\alpha \geq \frac{1}{3}$ & Gaussian widths\\
\hline
\end{tabular}
\end{small}
\end{center}
\vskip10pt
\caption{A catalogue of the results, with the main techniques used in their proofs.}
\label{tab:techniques}
\end{table}

In this section we give the proofs of the results described in the previous section. 
These results characterize the performance of the LSE for flows and the minimax rates
of convergence, and reveal a gap between the LSE and the best possible estimators for flows
on trees. In order to get sharp results, we employ a variety of proof techniques,
as shown in Table~\ref{tab:techniques}.

\subsection{Proof of Theorem~\ref{lse_upbd_shallow}: Upper bound on
  LSE for shallow trees}

A key result we use to prove our risk bounds for the LSE is the following technique
based on a quadratic Gaussian supremum functional, due to Sourav Chatterjee~\citep{Chat14}.
\begin{theorem}[Chatterjee, 2014]\label{chatthm}
Fix a rooted tree $T$ and a flow $\mu \in \F(T)$. Define the function $f_{\mu}: \R_{+}
\rightarrow \R$ as 
\begin{equation}\label{eq:chat}
  f_{\mu}(t)  := \E \left(\sup_{\nu \in \F: \|\nu - \mu\| \leq t} \left<Z, \nu - \mu \right> \right) - \frac{t^2}{2}
\end{equation}
where $Z$ is an $n$-dimensional mean zero Gaussian vector with covariance matrix $\sigma^2 I$.
Let $t^* > 0$ satisfy $f_{\mu}(t^*) \leq 0$. Then there
exists a  universal positive constant $C$ such that   
\begin{equation}\label{cim}
\E \frac{1}{n} \|\hat{\mu} - \mu\|^2 \leq \frac{C}{n} \max \left(t^{*\kern.01ex 2}, \sigma^2 \right).  
\end{equation}
\end{theorem}

The above theorem is a consequence of~\citet[Theorem 1.1]{Chat14} and~\citet[Proposition 1.3]{Chat14}. Let $t_{\mu}$ be the point in $[0, \infty)$ where $t \mapsto
f_{\mu}(t)$ attains its maximum; existence and uniqueness of
$t_{\mu}$ are proved in~\citet[Theorem 1.1]{Chat14}. Then there
exists a  universal constant $C$ such that   
\begin{equation*}
\E \frac{1}{n} \|\hat{\mu} - \mu\|^2 \leq \frac{C}{n} \max \left(t_{\mu}^2, \sigma^2 \right).  
\end{equation*}

This reduces the problem of bounding $\E \frac{1}{n} \|\hat{\mu} - \mu\|^2$ to that of bounding $t_{\mu}$. For this latter problem,~\citet[Proposition 1.3]{Chat14} observes that 
\begin{equation*}\label{chat2}
  \mbox{$t_{\mu} \leq t^{*}$ whenever $t^{*} > 0$ and $f_{\mu}(t^{*}) \leq 0$}.  
\end{equation*}
In order to bound $t_{\mu}$, one therefore seeks $t^{*} > 0$ such
that $f_{\mu}(t^{*}) \leq 0$. This now requires a bound on the
expected supremum of the Gaussian process in the definition of
$f_{\mu}(t)$ in  \eqref{eq:chat}.

\begin{proof}[Proof of Theorem~\ref{lse_upbd_shallow}]

Fix any $\mu \in \F(T_n)$. Let us denote by $g_\mu(t)$ the supremum
$$g_{\mu}(t) = \sup_{\nu \in \F(T_n): \|\nu - \mu\| \leq t} \langle Z,
\nu - \mu \rangle.$$ Denote by $L_i$ the set of nodes in level $1 \leq i \leq h$
in the tree $T_n$. Then we can write
\begin{align*}
g_{\mu}(t) &\leq \sup_{\nu \in \F(T_n): \|\nu - \mu\| \leq t} Z_1
(\nu_1 - \mu_1) + \sum_{i = 2}^{h}\; 
\sup_{\nu \in \F(T_n): \|\nu - \mu\| \leq t} \;\sum_{j \in L_i} Z_j (\nu_j - \mu_j) \\
&\leq t |Z_1| + \sum_{i = 2}^{h} \; \sup_{\nu \in \F(T_n): \|\nu - \mu\| \leq t} \;\Bigl(\sum_{j \in L_i} Z_j \nu_j  +  \sum_{j \in L_i} (-Z_j) \mu_j \Bigr).
\end{align*} 
The last inequality is because the constraint $\|\nu - \mu\| \leq t$
implies that $\nu_1 \leq \mu_1 + t$. The flow constraints imply,
within each level $L_i,$ that the sum $\sum_{j \in L_i} \nu_j$ is at
most $\mu_1 + t$. 
Therefore, the last equation can be further bounded to obtain
\begin{align*}
g_{\mu}(t) &\leq t |Z_1| + \sum_{i = 2}^{h} \bigl(\max_{j \in L_i} Z_j\bigr) (\mu_1 + t)  +  \bigl(\max_{j \in L_i} -Z_j\bigr) \mu_1
\end{align*}
where we have used the fact that $\sum_{i = 1}^{d} a_i b_i \leq (\sum_{i = 1}^{d} b_i) \max_{1 \leq i \leq d} a_i$ for any arbitrary vector $a$ and any nonnegative vector $b$.

Now taking the expectation of both sides, and using the fact that the
expectation of the maximum of $d$ independent $N(0,\sigma^2)$ random
variables is upper bounded by $\sigma \sqrt{2 \log d}$ and $E |Z_1| =
\sigma \sqrt{\frac{2}{\pi}}$,
we obtain
\begin{align*}
\E g_{\mu}(t) & \leq t \sigma \sqrt{\frac{2}{\pi}} + \sum_{i = 2}^{h}
(2 \mu_1 + t) \sigma \sqrt{2 \log |L_i|} \\
& \leq t \sigma \sqrt{\frac{2}{\pi}} + h (2 \mu_1 + t) \sigma \sqrt{2 \log \frac{n}{h}},
\end{align*}
where we have used Jensen's inequality in the second. 
Recalling that $f_{\mu}(t) = \E g_{\mu}(t) - t^2/2$, this implies that
$f_{\mu}(t^*) \leq 0$ where
\begin{equation*}
t^* = C \max\left\{\sigma h \sqrt{\log \frac{n}{h}},\; \sqrt{\mu_1 \sigma h \sqrt{\log \frac{n}{h}}}\right\}.
\end{equation*}
Theorem~\ref{chatthm} now finishes the proof of the upper bound.
\end{proof}

\subsection{Proof of Theorem~\ref{simplexminimax}: Minimax lower bound
for shallow trees}

To prove this minimax lower bound we use the following standard
version of Fano's lemma. For any set $\A \subset \R^n$ define a
$\delta$-packing set of $\A$ to be a finite subset of $\A$ such that
for any two distinct points $a \neq a' \in \A$ we have 
separation $\|a - a'\|^2 \geq \delta$.

\begin{lemma}[Fano]\label{fano}
If $W$ is a $\delta$-packing set of $\Theta$
then 
\begin{equation*}
\inf_{\hat{\theta}} \sup_{\theta \in \Theta} \E \|\hat{\theta} - \theta\|^2 
\geq \frac{\delta}{2} \Bigl(1 - \frac{\Delta + \log 2}{\log
  |W|}\Bigr)
\end{equation*}
where $\Delta = \max_{w \neq w'} \text{KL}(\P_{w},\P_{w'})$,
with $\P_w$ denoting the distribution of the data $Y$ when the true underlying parameter is $w$.
\end{lemma}

We also require the following version of the Varshamov-Gilbert coding lemma. 
\begin{lemma}[Varshamov-Gilbert]\label{vg}
Let $k$ and $d$ be two positive integers with $k \leq d/6$.
Then there exists a set $W \in \{0,1\}^d$ such that the following three conditions hold:
\begin{enumerate}
\item For any $w \in W$ we have $\sum_{i = 1}^{d} w_i = k$.
\vskip8pt
\item For any $w \neq w' \in W$ we have $\sum_{i = 1}^{d} \I\{w_i \neq w'_i\} \geq k/2$.
\vskip8pt
\item $\log |W| \geq \frac{k}{2} \log(1 + \frac{d}{4k})$.
\end{enumerate}
\end{lemma}


\begin{proof}[Proof of Theorem~\ref{simplexminimax}]
We first prove the minimax lower bound for the star graph with one root and $n$
children, with the parameter space $\F_{n,V} \subset \R^n$ being
the space of flows on the star graph with root value at most $V$. 
For any fixed $k \leq n/2$ consider the packing set $W$ guaranteed by
Lemma~\ref{vg}. For each element $w\in W$, we obtain a flow on the
star graph by setting the root flow to be $V$, and setting the
children flows to $V/k \cdot w \in \reals_+^n$.  This
forms a packing set of $\F_{n,V}$ of squared radius $V^2/2k$. The
Euclidean distance between any two distinct
elements in this packing set is at most $2V^2/k$. Therefore an
application of Lemma~\ref{fano} 
gives the minimax lower bound
\begin{equation*}
\inf_{\tilde{\theta}} \sup_{\theta \in \F_{n,V}} \E \|\tilde{\theta} -
\theta\|^2 \geq 
\max_{1 \leq k \leq n/6} \left\{\frac{V^2}{8k} \Bigl(1 - \frac{V^2/(2k \sigma^2) + \log 2}{\frac{k}{2} \big(\log(1 + \frac{n}{4k})\big)}\Bigr)\right\}.
\end{equation*}
By choosing $k$ so that $1 \leq k \leq n/60$ we can ensure
\begin{equation*}
\frac{\log 2} {\frac{k}{2} \big(\log(1 + \frac{n}{4k})} \leq \frac{2 \log 2}{\log(16)} \leq \frac{1}{2}.
\end{equation*}
and thus obtain the somewhat cleaner bound
\begin{equation}\label{basiclb}
\inf_{\tilde{\theta}} \sup_{\theta \in \S_{V}} \E \|\tilde{\theta} -
\theta\|^2 \geq 
\max_{1 \leq k \leq n/60} \left\{\frac{V^2}{8k} \Bigl(\frac{1}{2} - \frac{V^2}{k^2 \sigma^2 \big(\log(1 + \frac{n}{4k})\big)}\Bigr)\right\}.
\end{equation}

We now select $k$ differently according to the scale of the root flow
$V$.

\paragraph{\it\bfseries Small flow: $0 \leq V \leq \sigma \sqrt{\log n}$.}
In this case we set $k = c$ for some appropriate constant $c \geq 1$. Then we have
\begin{align*}
\frac{V^2}{k^2 \sigma^2 \big(\log(1 + \frac{n}{4k})\big)} = \frac{V^2}{c^2 \sigma^2 \big(\log(1 + \frac{n}{4c})\big)} \leq \frac{\sigma^2 \log n}{c^2 \sigma^2 \big(\log(1 + \frac{n}{4c})\big)} \leq 
\frac{1}{4}
\end{align*}
where the first inequality uses the bound on $V$ and the last
inequality holds for all large enough $n$ if $c$ is 
chosen to be a sufficiently large constant. 
Together with~\eqref{basiclb}, this implies a minimax lower bound of $V^2/c$.

\paragraph{\it\bfseries Large flow: $\sigma \sqrt{\log n} < V \leq n^{1 - \epsilon} \sigma$.}
Let $c_{\epsilon} \geq 1$ be a sufficiently large positive constant.   
Set $k = c_{\epsilon} \frac{V}{\sigma \sqrt{\log n}}$. Then $k \geq 1$
by the lower bound condition on $V$. 
By the upper bound condition, for sufficiently large $n$ we have
$k  \leq c_{\epsilon}\frac{n \sigma}{\sigma \log n} \leq \frac{n}{60}$.
Hence the choice of $k$ is valid for all large enough $n$, and we have
\begin{align*}
\frac{V^2}{k^2 \sigma^2 \big(\log(1 + \frac{n}{4k})\big)} = \frac{V^2 \sigma^2 \log n}{c_\epsilon V^2 \sigma^2 \big(\log(1 + \frac{n \sigma \sqrt{\log n}}{4 c_\epsilon V})\big)} \leq \frac{\log n}{c_{\epsilon} \big(\log(1 + n^{\epsilon} \log n)\big)} \leq \frac{1}{4} 
\end{align*}
where the first inequality uses the upper bound on $V$  and the last inequality is true for all large enough $n$. 
Using~\eqref{basiclb} then yields the
lower bound of $\frac{V^2}{k} = V \sigma \sqrt{\log n}$, up to a
constant factor depending on $\epsilon$. 
Note that our choice of $k$ need not be a integer; choosing the
nearest integer to it which will only change the lower bound by constants.

The proof for a general tree $T$ with height $h$ bounded by a constant $C$ is very
similar. In particular, there must exist a level $l$ of the tree with at
least $n/C$ elements. 
Define a flow where the values at level $l$ assume
the values given by a vector $w \in W\subset \{0,1\}^{n/C}$ multiplied
by $V/k$. Define the flow at levels $l' < l$ by defining the value at any vertex to be the sum of the values at its
children, and the flow at levels $l' > l$ to be zero. 
This defines a flow indexed by an element $w \in W$. Since the height of $T$
is at most a constant $C$, one can verify that this defines a packing set of
$\F_{n,V}$ of squared radius $V^2/k$ up to a constant factor. Also, for any two distinct
elements in this packing set, the Euclidean distance between them is
at most $V^2/k$ up to a constant factor. The remainder of the argument
using Fano's lemma
proceeds as above.
\end{proof}

\subsection{Proof of Theorem~\ref{lse_lb_shallow}: Lower bound for the
LSE on shallow trees}
We will need the following standard lemma about the expectation of the maxima of independent normal random variables. 
\begin{lemma}{\label{gaumax}}
Let $Z_1,\dots,Z_n$ be independent $N(0,\sigma^2)$ random
variables. Then 
\begin{equation*}
 \frac{\sqrt{2\sigma^2 \log n}}{\sqrt{\pi} \log 2} \leq \E \max_{1\leq
  i\leq n} Z_i \leq \sqrt{2\sigma^2 \log n}.
\end{equation*}
\end{lemma}

We also need the following lemma about the projection to a closed convex cone.
\begin{lemma}\label{cone}
Let $K \subset \R^n$ be a closed convex cone, and denote
by $\Pi_K(z)$ the projection of $z\in \R^n$ to $K$. Then 
\begin{equation}
\|\Pi_K(z)\| = \sup_{\nu \in K: \|\nu\| \leq 1} \langle z,\nu \rangle.
\end{equation}
\end{lemma}

\begin{proof}
By definition, $\Pi_K(z)$ maximizes $\langle z,v \rangle -
\frac{1}{2}\|v\|^2 $ among all $v \in K$; thus
\begin{align*}
\|\Pi_{K}(z)\| & = \argmax_{t \geq 0} \Big\{\sup_{\theta \in K: \|\theta\| \leq t} \langle z, \theta \rangle - \frac{t^2}{2}\Big\}\\
& = \argmax_{t \geq 0}  \Big\{t \sup_{\theta \in K: \|\theta\| \leq 1} \langle z, \theta \rangle - \frac{t^2}{2}\Big\},
\end{align*}
where the second equation uses the fact that $K$ is a cone.
Computing the maximum of the quadratic finishes the proof of the lemma.
\end{proof}

\begin{proof}[Proof of Theorem~\ref{lse_lb_shallow}]
We will prove the desired lower bound on the risk of the LSE at the
origin, where $\mu_i = 0$ for all $1 \leq i \leq n$. Since the space
of flows $\F(T_n)$ is a cone,  Lemma~\ref{cone}
implies that
\begin{equation*}
\|\hat{\mu} - \mu\| = \sup_{\nu \in \F(T_n): \|\nu\| \leq 1} \langle Z,\nu \rangle.
\end{equation*}
Therefore 
\begin{equation}\label{gwlb}
\left(\E \sup_{\nu \in \F(T_n): \|\nu\| \leq 1} \langle Z,\nu \rangle\right)^2 \leq \E \left(\sup_{\nu \in \F(T_n): \|\nu\| \leq 1} \langle Z,\nu \rangle\right)^2 = \E \|\hat{\mu} - \mu\|^2.
\end{equation}
Thus, it suffices to lower bound the Gaussian width term $\E \sup_{\nu
  \in \F(T_n): \|\nu\| \leq 1} \langle Z,\nu \rangle$. Since the tree
$T_n$ has height $h$, there exists a level $l$ with no fewer than
$n/h$ vertices. Let us denote this set of vertices by $L$, and define
\begin{equation*}
i^* = \argmax_{i \in L} Z_i.
\end{equation*}
There is a unique path from the root to $i^*$; define $\hat{\nu}$ to
be equal to $1/\sqrt{h_n}$ on this path and equal to zero off the
path. Then $\hat{\nu} \in \F(T_n)$ with $\|\hat{\nu}\| \leq 1$. Therefore we can write
\begin{equation*}
\sup_{\nu \in \F(T_n): \|\nu\| \leq 1} \langle Z,\nu \rangle \geq \langle Z,\hat{\nu} \rangle = \frac{1}{\sqrt{h_n}} \max_{i \in L} Z_i + \frac{1}{\sqrt{h_n}} \sum_{i = 1}^{l - 1} Z_i
\end{equation*}
where $Z_i$ are the error vector coordinates as we traverse from the
root to the vertex $i^*$. 
Taking expectation and applying Lemma~\ref{gaumax}, we have
\begin{equation*}
\E \sup_{\nu \in \F(T_n): \|\nu\| \leq 1} \langle Z,\nu \rangle \geq \E \langle Z,\hat{\nu} \rangle = \frac{1}{\sqrt{h_n}} \E \max_{i \in L} Z_i \geq C \frac{\sigma}{\sqrt{h_n}} \sqrt{\log \frac{n}{h_n}}.
\end{equation*}
This inequality combined with \eqref{gwlb} finishes the proof of the lower bound.
\end{proof}

\subsection{Proof of Theorem~\ref{isotonicupbd}: Isotonic upper bound on the
  LSE for deep trees}\label{sec:isotonicupb} We first set up some
notation. Recall that for a given parameter $0 \leq \alpha \leq 1$ the
root of the tree $\T_{\alpha,n}$ has $m = \lceil n^\alpha\rceil$
children. Each of these $m$ children is the starting point of a path
of length $l = \lceil n^{1 - \alpha} \rceil$. Clearly there are $lm +
1 = \Theta(n)$ vertices in $\T_{\alpha,n}$. For simplicity, we will
assume $lm + 1 = n$. The set of flows on $\T_{\alpha,n}$ is denoted
by $\F(\T_{\alpha,n}) := \F_{n,\alpha}$ for notational simplicity. The
set $\F_{n,\alpha}$ is a closed convex cone of $\R^n$. For
convenience, we will index the components of a flow $\mu$ in
$\F_{n,\alpha}$ as shown in Figure~2, with the root flow denoted
by $\mu_{1}$ and the values $\mu_1^{(1)}, \mu_1^{(2)}, \ldots,
\mu_1^{(m)}$ denoting the flows to the $m$ children of the root;
thus $\mu_1 \geq \mu_1^{(1)} + \mu_1^{(2)} + \cdots \mu_1^{(m)}$.  The
monotonic nondecreasing flow along the $j$th path is then $\mu_1^{(j)}
\geq \mu_2^{(j)} \geq \cdots \mu_l^{(j)}$. Sometimes we will denote
the vector $(\mu_1^{(j)},\mu_2^{(j)},\mu_l^{(j)})$ by $\mu^{(j)}$. For
any vector $\theta \in \R^n$ let $k(\theta)$ denote the cardinality of
the set $\{\theta_1,\dots,\theta_n\}$. 

We will need the following
lemma about approximating a monotone sequence by a piecewise constant
sequence.

\begin{lemma}[Approximation]\label{approx}
Let $\epsilon > 0$. Fix any nonincreasing sequence $\theta \in
\R^n$. Let $V = \theta_{1} - \theta_{n}$. Then there exists a
nonincreasing sequence $\alpha \in \R^n$ with $\alpha_1 \leq \theta_1$
such that $k(\alpha) \leq \frac{V}{\epsilon} + 1$ and $\|\theta -
\alpha\|^2 \leq n \epsilon^2$.
\end{lemma}

\begin{proof}
Define $s_0 = 1$. For $1 \leq i \leq \lfloor \frac{V}{\epsilon} \rfloor$ recursively define
\begin{equation*}
s_i = \min\bigl\{i: \theta_{s_{i - 1}} - \theta_i > \epsilon\bigr\}.
\end{equation*}
Define $s_{\lfloor\frac{V}{\epsilon}\rfloor + 1} = n$. Now define $\alpha$ as follows:
\begin{equation*}
\alpha_j = \sum_{i = 1}^{\lfloor\frac{V}{\epsilon}\rfloor + 1} \theta_{s_{i - 1}} \I\{s_{i - 1} \leq j < s_{i}\}.
\end{equation*}
It is clear that $\alpha_n \leq \theta_n$ and $k(\alpha) \leq \frac{V}{\epsilon} + 1$. Also by definition, we have $|\alpha_j - \theta_j| \leq \epsilon$ for all $1 \leq j \leq n$ which finishes the proof of the lemma.
\end{proof}

For any closed convex cone $C \subset \R^n,$ denote the projection
operator onto $C$ by $\Pi_{C}$ and define 
\begin{equation}\label{statdim}
\delta(C) = \E \|\Pi_{C} Z\|^2
\end{equation}
where $Z \sim N(0,I)$. The quantity $\delta(C)$ is called the
statistical dimension of $C$ and 
generalizes the concept of dimension of a subspace (see~\cite{LivEdge}). The tangent cone to $C$ at $u\in\reals^n$ is
defined by 
\begin{equation}\label{tancone}
\T_{C,u} = \{v - tu: v \in C, \,t \geq 0\}. 
\end{equation}
The tangent cone $\T_{C,u}$ is a closed convex cone of $\R^n$ and
hence one can talk about the statistical dimension 
$\delta(\T_{C,u})$.

The following oracle risk bound is due to \cite{bellec2015sharp}.
\begin{lemma}[Bellec, Proposition 2.1]\label{bellec}
Let $C\subset\reals^n$ be a closed convex cone,
and let $\hat{\mu}$ denote the least squares estimate of
$\mu\in C$,  that is, $\hat{\mu} = \Pi_{C}(\mu + Z)$ where $Z \sim N(0,\sigma^2 I)$. Then we have the pointwise inequality
\begin{equation*}
\|\hat{\mu} - \mu\|^2 \leq \|u - \mu\|^2 + \sigma^2 \bigl\|\Pi_{\T_{C,u}} ({Z}/{\sigma})\bigr\|^2
\end{equation*}
for any $u \in C$. 
As a consequence, 
\begin{equation*}
\E \|\hat{\mu} - \mu\|^2 \leq \|u - \mu\|^2 + \sigma^2 \delta(\T_{C,u}).
\end{equation*}
\end{lemma}

To use this result we need to bound the statistical dimension of the tangent cone to the space of flows.
\begin{lemma}\label{tanconecharac}
Fix any $\nu \in \F_{n,\alpha}$, and let $k^{(i)} = k(\nu^{(i)})$ for
$1 \leq i \leq m$
be the number of steps along the $i$th path. Then we 
have the following upper bound on the statistical dimension for the tangent cone $\T_{\F_{n,\alpha},\nu}$:
\begin{equation*}
\delta(\T_{\F_{n,\alpha},\nu}) \leq \log(en) \sum_{i = 1}^{m} k^{(i)}.
\end{equation*}
\end{lemma}

\begin{proof}
Let $B^{(i)}_1,\dots,B^{(i)}_{k^{(i)}}$ denote the contiguous blocks where $\nu^{(i)}$ is constant. For any vector $\theta \in \R^n$ and any $A \subset \{1 \leq i \leq n\}$ denote by $\theta_{A}$ the vector $\theta$ with coordinates restricted to be in the set $A$. Also denote the cone of nonincreasing sequences in $\R^d$ by $\M_{d}$.
Equipped with this notation we now claim that
\begin{equation}\label{claim}
\T_{\F_{n,\alpha},\nu} \subset \K = \left\{\theta \in \R^n: \theta_{B^{(i)}_j} \in \M_{|B^{(i)}_j|}\:\:\text{for all}\:\: 1 \leq i \leq m, 1 \leq j \leq k^{(i)}\right\}.
\end{equation}
Assuming this claim for now, by monotonicity of statistical dimension, we have
\begin{equation*}
\delta(\T_{\F_{n,\alpha},\nu}) \leq \delta(\K).
\end{equation*}
Since $\K$ is a cone composed of disjoint monotone pieces and the statistical dimension of the monotone cone is known to be $\delta(\M_{d}) \leq \log(ed)$ we have
\begin{equation*}
\delta(\T_{\F_{n,\alpha},\nu}) \leq \delta(\K) \leq \sum_{i = 1}^{m} \sum_{j = 1}^{k^{(i)}} \log(e|B^{(i)}_j|) \leq \sum_{i = 1}^{m} k^{(i)} \log(en).
\end{equation*}
It remains to prove~\eqref{claim}. Take any element $\theta \in
\T_{\F_{n,\alpha},\nu}$. Then by~\eqref{tancone} there exists $t > 0$
and $v \in \F_{n,\alpha}$ such that $\theta = v - t \nu$. Now consider
any block $B^{(i)}_j$. By definition, $\nu_{B^{(i)}_j}$ is a constant
vector and $v_{B^{(i)}_j} \in \M_{|B^{(i)}_j|}$. This implies
$\theta_{B^{(i)}_j} \in \M_{|B^{(i)}_j|}$,
which proves~\eqref{claim}. 
\end{proof}

We are now ready to prove Theorem~\ref{isotonicupbd}.

\begin{proof}[Proof of Theorem~\ref{isotonicupbd}]
Fix an arbitrary $\epsilon > 0$. Recall $l = \lceil n^{1 - \alpha} \rceil$ and $m = \lceil n^{\alpha} \rceil$. For each path $\mu^{(i)} = (\mu^{(i)}_1,\dots,\mu^{(i)}_l)$ we can use Lemma~\ref{approx} to obtain a nonincreasing sequence $\nu^{(i)} = (\nu^{(i)}_1,\dots,\nu^{(i)}_l)$ such that 
\begin{equation}\label{pieces}
k(\nu^{(i)}) \leq \frac{\mu^{(i)}_1}{\epsilon} + 1
\end{equation}
along with $\|\mu^{(i)} - \nu^{(i)}\|^2 \leq l \epsilon^2$ and
$\nu^{(i)}_{1} \leq \mu^{(i)}_{1}$. Also define $\nu_{1} =
\mu_{1}$. Then it is clear that $\nu \in \F_{n,\alpha}$. Using Lemma~\eqref{bellec} we deduce
\begin{equation*}
\E \|\hat{\mu} - \mu\|^2 \leq \|\nu - \mu\|^2 + \sigma^2 \delta(\T_{\F_{n,\alpha},\nu}).
\end{equation*}
Now we have $\|\nu - \mu\|^2 = \sum_{i = 1}^{m} \|\nu^{(i)} -
\mu^{(i)}\|^2 \leq ml \epsilon^2 = n \epsilon^2$. Therefore
using~\eqref{pieces} and Lemma~\ref{tanconecharac} we can conclude that
\begin{align*}
\E \|\hat{\mu} - \mu\|^2 \leq \:&n \epsilon^2 + \sigma^2 \sum_{i =
  1}^{m} \big(\frac{\mu^{(i)}_1}{\epsilon} + 1\big) \log(en) \\
& n \epsilon^2 + \sigma^2 \sum_{i = 1}^{m}
\big(\frac{\mu^{(i)}_1}{\epsilon}\big) \log(en) + \sigma^2 m \log(en)
\\
& \leq n \epsilon^2 + \sigma^2 \frac{\mu_1}{\epsilon} \log(en) + \sigma^2 m \log(en).
\end{align*}
By choosing $\epsilon = \big(\frac{\sigma^2 \mu_{1}
  \log(en)}{n}\big)^{1/3}$ 
we obtain the risk bound
\begin{align*}
\E \|\hat{\mu} - \mu\|^2 \leq 2 n^{1/3} \big(\sigma^2 \mu_{1} \log(en)\big)^{2/3} + \sigma^2 m \log(en),
\end{align*}
which finishes the proof of the theorem.
\end{proof}

\subsection{Proof of Theorem~\ref{simplexupbd}: Simplex upper bound on the
  LSE for deep trees}\label{sec:simplexupbd}
We begin by establishing some notation. If $(M, \rho)$ is a metric
space, a set $\hat T \subset M$ is called an $\eps$-cover of $T \subset M$
in case 
\begin{equation}
T \subset \cup_{t \in \hat{T}} B_{\eps}(t)
\end{equation} 
where $B_{\eps}(t)
= \{x \in M: \rho(x,t) < \eps\}$ denotes the ball of radius $\eps >0$
centered at $t \in M$. The $\eps$-covering number of $T$ is the
cardinality of the smallest cover of $T$:
\begin{equation} 
N(\eps,T,\rho) =
\min \left\{|\hat T|: \hat T \mbox{ is an } \eps\mbox{-cover of }
T\right\}.
\end{equation}
The metric will almost always be given by the usual
Euclidean norm, in which case we denote the covering numbers by
$N(\epsilon,T)$; otherwise, the metric will be explicitly mentioned.

An important result for us here again is Theorem~\ref{chatthm}. It 
This now requires a bound on the
expected supremum of the Gaussian process in the definition of
$f_{\mu}(t)$ in~\eqref{eq:chat}. The following chaining result gives an upper bound on the expected supremum of the required Gaussian process. This chaining result is
sometimes known as Dudley's entropy integral inequality; a proof for the version of the bound stated below can be found in~\citet[Lemma A.2]{chatterjee2015matrix}.
\begin{theorem}[Chaining]\label{dudthm}
For every $\mu \in \F(T)$ and $t > 0$,
{\begin{equation*}\label{gaup}
\E\left( \sup_{\nu \in \F(T): \|\nu - \mu\| \leq t}
     \langle \nu - \mu, Z \rangle\right)
 \leq  \sigma \inf_{0 < \delta
      \leq 2t} \left\{12 \int_{\delta}^{2t}
      \sqrt{\log N(\epsilon, B(\mu, t))} \;
     d\epsilon + 4 \delta \sqrt{n} \right\} .  
\end{equation*}}
\end{theorem}

The first step in the proof of Theorem~\ref{simplexupbd} is to upper bound the covering number for the metric space of flows.
\begin{lemma}\label{lem:cover}
Fix $V > 0$ and a positive integer $n$. For any tree $T$ with
$n$ nodes and depth $h$ denote by  $\F_{V}(T) =:\F_{V}$ the set of flows
on $T$ where the root flow $\mu_1$ is no greater than $V$. For any $\epsilon > 0$,
define 
\begin{equation}
\label{eq:meps}
m_{\epsilon} = \left\lceil \frac{V^2 h}{\epsilon^2} \right\rceil.
\end{equation}
Then 
\begin{equation}
\log N(\epsilon,\F_{V}) \leq m_{\eps} \left(1 + \log\left(1 + \frac{n}{m_{\eps}}\right)\right).
\end{equation}
\end{lemma}

\begin{proof}
Let $h$ denote the height of the tree, let $\LL = \{i : \C(i) =
\emptyset\} \subset [n]$ denote
the set of leaf nodes, and let $\II = [n] - \LL$ be the set
of internal, or non-leaf nodes.  We first note that a flow $\mu \in
\F_V$  can be uniquely identified by the collection of
\textit{leaks} at the nodes, with the leak at node $i$ defined by
\begin{equation}
l^\mu_i = \mu_i - \sum_{j\in\C(i)} \mu_j.
\end{equation}
If $i\in\LL$ is a leaf, then $l^\mu_i = \mu_i$ and
we also call this the \textit{residue} of the flow at $i$.
By definition, 
\begin{equation}
\label{eq:leakeq}
\sum_{i\in\II\cup\LL} l^{\mu}_i \leq \mu_1;
\end{equation}
that is, the residues and leaks can together be no greater than the flow into the root node.

For any positive integer $m \in \Z_{+},$ we define a set of flows $\F_{V}^{m} \subset \F_{V}$ as
\begin{equation}
\F_{V}^{m} = \left\{\mu \in \F_{V}: \frac{m}{V} l^{\mu} =
(i_1,\dots,i_n),\;
\text{for some $(i_1,\dots,i_n) \in \Z_{+}^n$ with $i_1 + \dots + i_n \leq m$}\right\}.
\end{equation}

We now show using a probabilistic argument that $\F_{V}^{m}$ is a
$\left(V^2 h\right)/m$ covering set for $\F_{V}$.  Fixing a flow $\mu \in
\F_{V}$, we define a random flow $F$, whose distribution depends on
$\mu$, by specifying its leaks $l^F$ as follows:
\begin{equation}
l^F = \begin{cases}
 V e_i & \text{ with probability  $\displaystyle \frac{l^\mu_i}{V}$} \\[10pt]
  0    & \text{ with probability $\displaystyle 1- \frac{1}{V} \sum_{i} l_i^\mu$}.
\end{cases}
\end{equation}
Here $e_i$ denotes the $n$ dimensional vector with $1$ in the $i$th
entry and $0$ everywhere else. 
For any positive integer $m$, we now let 
$F_1,\dots,F_m$ be $m$ i.i.d. copies of $F$, and define the
mean flow
\begin{equation}
\overline{\mu}^m = \frac{1}{m} \left(F_1 + \cdots + F_m\right).
\end{equation}
Note that
\begin{equation}\label{ineq4}
\overline{\mu}^{m} \in \F_{V}^{m}\;\;\text{with probability $1$}.
\end{equation}

Consider now the expected Euclidean distance between $\mu$ and the random flow $\overline{\mu}^{m}$. We can write
\begin{align}
\E \|\overline{\mu}^{m} - \mu\|^2 &= 
\sum_{i=1}^n \E (\overline{\mu}^{m}_{i} -
\mu_i)^2 \\
&= \sum_{i} \Var (\overline{\mu}^{m}_{i}),
\label{ineq1}
\end{align}
where $\Var$ refers to the variance of a random variable.  The last
equality holds because $\E\overline{\mu}^m = \mu$; that is, $\overline{\mu}^m$ is unbiased for $\mu$.

Let $\overline{l}^m_{v}$ denote the leak of the flow $\overline{\mu}^{m}$
at node $v$.  Fix a leaf node $v \in \LL$. 
Then the leak $\overline{l}^m_v$ is a mean
of $m$ Bernoulli random variables, each taking the value $V$ with
probability $\frac{\mu_{v}}{V}$ and $0$ with the complementary
probability. Hence we have that
\begin{equation}\label{eq:ineq7}
\Var(\overline{\mu}^{m}_{v}) = \frac{V^2}{m} \frac{\mu_{v}}{V} \left(1
- \frac{\mu_{v}}{V}\right) \leq \frac{V \mu_v}{m}.
\end{equation}
More generally, we have that for any node $v$,
\begin{equation}
\overline{\mu}^{m}_v = \sum_{k \in \C(v)} \overline{\mu}^{m}_{k} +
\overline{l}^m_{v}
\end{equation} 
and hence
\begin{align}
\Var(\overline{\mu}^{m}_{v}) = \Var\left(\sum_{k \in \C(v)}
\overline{\mu}^{m}_{k} + \overline{l}^m_{v}\right) \leq 
\sum_{k \in \C(v)} \Var\left(\overline{\mu}^{m}_{k}\right) + \Var(\overline{l}^m_{v}),
\end{align}
where the inequality holds since the random variables
$\left(\overline{\mu}^{m}_{1},\dots,\overline{\mu}^{m}_{1},\overline{l}^m_{v}\right)$
are pairwise negatively correlated, by construction of the random flow
$\overline{\mu}^{m}$. Applying this argument recursively,
we have
\begin{equation}\label{ineq2}
\Var(\overline{\mu}^{m}_{v})  \leq \sum_{u \in \Sub(v)} \Var(\overline{l}^m_{u}),
\end{equation}
where $\Sub(v)$ denotes the subtree rooted at $v$.
For any non-leaf $i \in\II$, by similar reasoning as in~\eqref{eq:ineq7}, we have 
\begin{equation}
\Var(\overline{l}^m_{i}) \leq \frac{V}{m} l_{i}^\mu.
\end{equation}
Using this observation together with~\eqref{ineq1} and~\eqref{ineq2},
and denoting by $d_i \leq h$ the depth of node $i$, we conclude
\begin{align}
\E \|\overline{\mu}^{m} - \mu\|^2 
&= \sum_{i} \Var (\overline{\mu}^{m}_{i})\\
&\leq \sum_{i} d_i
\Var(\overline{l}^m_i) \\
&\leq \frac{V}{m} \sum_{i} d_i {l}_i^\mu \\
&\leq \frac{Vh}{m} \sum_{i} {l}_i^\mu \\
&\leq \frac{V^2 h }{m}.
\end{align}

The above result holds in expectation with respect to the random draw
$\overline{\mu}^{m}$, which by~\eqref{ineq4} always lies in the finite
set $\F_{V}^{m}$. Hence, we can assert the existence of an element
$\tilde{\mu} \in \F_{V}^{m}$ such that
\begin{equation}\label{ineq5}
\|\mu - \tilde{\mu}\|^2 \leq \frac{V^2 h}{m}.
\end{equation}
Since $\mu \in \F$ was arbitrarily chosen, we have shown that $\F_{V}^{m}$ is a $\left(V^2 h\right)/m$ covering set for $\F_{V}$.

Finally, note that the set $\F_{V}^{m}$ is in one-to-one correspondence with
the set $\{(i_1,\dots,i_n) \in \Z_{+}^n: i_1 + \dots + i_n \leq m\}$.
By standard combinatorics, we have
\begin{equation}\label{ineq6}
|\F_{V}^{m}| = \binom{n + m}{m} \leq  \left(\frac{e (n +
  m)}{m}\right)^m.
\end{equation}
For any given $\epsilon > 0$ we then choose $m = m_{\eps}$ 
according to \eqref{eq:meps} to deduce the statement of the lemma.
\end{proof}

\begin{remark}\rm
The idea of the above proof to demonstrate covering sets by
randomization is not new, and is sometimes referred to as ``Maurey's
argument.'' Our proof is a generalization of the proof in \citet[Lemma
  2.6.11]{vaartwellner96book} to the setting of flows on rooted trees.
\end{remark}


We are now ready to prove Theorem~\ref{simplexupbd}.
\begin{proof}[Proof of Theorem~\ref{simplexupbd}]
Fix $\mu \in \F$.  We use Theorem~\ref{chatthm} 
by first upper bounding the function
\begin{equation}
f_{\mu}(t)  := \E \left(\sup_{\nu \in \F: \|\nu - \mu\| \leq t} \left<Z, \nu - \mu \right> \right) - \frac{t^2}{2} 
\end{equation}
where $Z \sim N(0, \sigma^2 I_{n \times n})$, using Dudley's entropy
integral inequality. Note that the diameter of the set $\{\nu \in \F:
\|\nu - \mu\| \leq t\}$ is at most $2t$. 
Setting $\delta = 1/\sqrt{n}$ in Theorem~\ref{dudthm} we obtain 
\begin{align}
\E\left( \sup_{\nu \in \F: \|\nu - \mu\| \leq t} \left<Z,\nu - \mu\right>\right) &
\leq 4 \sigma + 12 \sigma \int_{\frac{1}{\sqrt{n}}}^{2t} \sqrt{\log
  N(\eps,\{\nu \in \F: \|\nu - \mu\| \leq t\})} d\eps \\
&\leq 4 \sigma + 12 \sigma \int_{\frac{1}{\sqrt{n}}}^{2t} \sqrt{\log N(\eps,\{\nu \in \F: \nu_{1} \leq \mu_{1} + t\})} d\eps
\end{align}
where the second inequality follows from the inclusion
\begin{equation}
\{\nu \in \F: \|\nu - \mu\| \leq t\} \subset \{\nu \in \F: \nu_{1} \leq \mu_{1} + t\}.
\end{equation}
Now we are in a position to use Lemma~\ref{lem:cover} since it gives
us upper bounds on log covering numbers of sets of the form $\{\nu \in
\F: \nu_{1} \leq \mu_{1} + t\}$. In particular, we have
\begin{align}
\nonumber
\log N(\eps,\{\nu \in \F: &\nu_{1} \leq \mu_{1} + t\})\\
 & \leq \left(1 +
\frac{(\mu_{1} + t)^2 h}{\eps^2}\right) \left[1 + \log\left(1 +
  \frac{n}{1 + (\mu_{1} + t)^2 h/\eps^2}\right)\right] \\
 & \leq \frac{(\mu_{1} + t)^2 h}{\eps^2} \left[1 + \log\left(1 + \frac{n
    \eps^2}{(\mu_{1} + t)^2 h}\right)\right] + (1 + \log(n + 1))
\end{align}
where we have used the facts that $\lceil x \rceil \leq x + 1$ and 
$x\left(1 + \log(1 + c/x)\right)$  is a nondecreasing function of
$x\in \R_{+}$ for $c> 0$. 
The simple inequality $\sqrt{a + b} \leq \sqrt{a} + \sqrt{b}$ now gives
\begin{equation}
\E\left( \sup_{\nu \in \F: \nu_{1} \leq \mu_{1} + t} \left<Z,\nu -
\mu\right>\right)
 \leq 4 \sigma + 12 \sigma \left((\mu + t) \sqrt{h} I + 2t \sqrt{1 + \log(n + 1)}\right)
\end{equation} 
where $I$ is the integral
\begin{align*}
I &= \int_{0}^{2t}
\frac{1}{\epsilon} \left(1 + \log\left[1 + \frac{n \eps^2}{(\mu_1 + t)^2
    h}\right]\right)^{1/2} d\eps\\
& \leq C\left(1 + \log\left[1 + 4\frac{n t^2}{(\mu_1 + t)^2 h}\right]\right)^{3/2}
\end{align*}
where the inequality follows from elementary calculus.
We thus have
\begin{align}
\nonumber 
\E\Bigl( &\sup_{\nu \in \F: \nu_{1} \leq \mu_{1} + t} \left<Z,\nu - \mu\right> \Bigr)\\ 
& \leq 4 \sigma + 12 \sigma C (\mu_1 + t) \sqrt{h}
\left(1 + \log\left[1 + 4\frac{n t^2}{(\mu_1 + t)^2 h}\right]\right)^{3/2} + 24
\sigma t \sqrt{1 + \log(n + 1)} \\
& \leq 4 \sigma + 12 \sigma C (\mu_1 + t) \sqrt{h} \left(1 +
\log\left[1 + 4n\right]\right)^{3/2} 
+ 24 \sigma t \sqrt{1 + \log(n + 1)}
\end{align} 
where the second inequality uses 
${t^2}/{(\mu_1 + t)^2} \leq 1$ and $h \geq 1$.
Therefore, 
\begin{equation}\label{fmut}
f_{\mu}(t) \leq 4 \sigma + 12 \sigma C (\mu_1 + t) \sqrt{h} \left(1 +
\log\left[1 + 4n\right]\right)^{3/2} + 24 \sigma t \sqrt{1 + \log(n +
  1)} - \frac{t^2}{2} \equiv g(t).
\end{equation}
Letting $t^{*}$ be the larger root of the quadratic function $g(t)$, we have 
\begin{equation}
f_{\mu}(t^{*}) \le g(t^{*}) = 0.
\end{equation} 
After some algebraic manipulation one can upper bound $t^{*2}$
according to 
\begin{equation*}
t^{*2} \leq C \sigma^2 h (1 + \log n)^3 + C \sigma \mu_1 \sqrt{h} (1 + \log n)^{3/2}.
\end{equation*}
Therefore equation~\eqref{cim} of Theorem~\ref{chatthm} implies that
\begin{equation}
\E \|\hat{\mu} - \mu\|^2 \leq C\left(\sigma^2 h (1 + \log n)^3 + \sigma \mu_1 \sqrt{h} (1 + \log n)^{3/2} \right),
\end{equation}
which completes the proof of the theorem.
\end{proof}

\subsection{Proof of Theorem~\ref{isominimax}: Minimax lower bound for
monotone sequences}
We shall use Assouad's lemma to prove Theorem~\ref{isominimax}. The following version of Assouad's Lemma is a consequence of Lemma 24.3 of~\cite{van2000asymptotic}.  
\begin{lemma}[Assouad]\label{suad}
Fix $V > 0$ and a positive integer $d$. Suppose that, for each $\tau \in \{-1, 1\}^d$, there is an associated $g^{\tau} \in \M_{V}$. Then 
  \begin{equation*}
\inf_{\tilde{\mu}} \sup_{\mu \in \M_{V}} \E \|\tilde{\mu} - \mu\|^2 \geq \frac{d}{8} \min_{\tau \neq \tau'} \frac{\|{g}^{\tau} - {g}^{\tau'}\|^2}{\ham(\tau,\tau')} \min_{\ham (\tau, \tau') = 1} \left(1 - \|\P_{{g}^{\tau}} - \P_{{g}^{\tau'}}\|_{TV} \right), 
  \end{equation*}
where $\ham(\tau, \tau') := \sum_{i=1}^d I\{\tau_i \neq \tau'_i\}$ denotes the Hamming distance between $\tau$ and $\tau'$ and $\|\cdot \|_{TV}$ denotes the total variation distance. The notation $\P_{{g}}$ for $g \in \M_{V}$ refers to the joint distribution of $y_{i} = g_{i} + \epsilon_{i}$, for $1 \leq i \leq n$ when $\epsilon_{i}$ are independent normally distributed random variables with mean zero and variance $\sigma^2$. 
\end{lemma}

\begin{proof}[Proof of Theorem~\ref{isominimax}]
Fix any integer $1 \leq k \leq n$ and define $m = \lfloor n/k\rfloor$. Define the vector $\mu^*_i = Vi/n$. For any $\tau \in \{0,1\}^k$ define the vector $\mu^{\tau}$ in the following manner:
\begin{equation}\label{assouadpacking}
\mu^{\tau}_j =
\begin{cases}
\mu^*_{(i - 1)m + 1}& \mbox{if} \:\:\tau_i = 0, (i - 1)m < j \leq im \\ \mu^*_{im}& \mbox{if} \:\:\tau_i = 1, (i - 1)m < j \leq im.
\end{cases}
\end{equation}
If $j > mk,$ then define $\mu^{\tau}_j = V$ for any $\tau \in \{0,1\}^k$. It is clear that $\mu^{}\tau \in \M_{V}$. Now note
\begin{equation*}
\|{\mu}^{\tau} - {\mu}^{\tau'}\|^2 = \frac{V^2}{n^2} \ham(\tau,\tau') (1^2 + 2^2 + \dots + m^2).
\end{equation*}
It is now not hard to see that
\begin{equation}\label{key}
\frac{V^2 m^3}{3 n^2} \ham(\tau,\tau') \leq \|{\mu}^{\tau} - {\mu}^{\tau'}\|^2 = \frac{V^2 m^3}{n^2} \ham(\tau,\tau').
\end{equation}
Also by Pinsker's inequality we have 
\begin{equation*}
(\|\P_{{\mu}^{\tau}} - \P_{{\mu}^{\tau'}}\|_{TV})^2 \leq \frac{1}{2} \textit{KL}(\P_{\mu^{\tau}},\P_{\mu^{\tau'}}) = \frac{1}{4 \sigma^2} \|{\mu}^{\tau} - {\mu}^{\tau'}\|^2 \leq \frac{V^2 m^3}{4 \sigma^2 n^2} \ham(\tau,\tau'),
\end{equation*}
where we used~\eqref{key} in the last inequality. An application of
Assouad's Lemma, along with the last two equations, now yields the minimax lower bound
\begin{equation*}
\inf_{\tilde{\mu}} \sup_{\mu \in \M_{V}} \E \|\tilde{\mu} - \mu\|^2 \geq \frac{k}{8} \frac{V^2 m^3}{3 n^2} \big(1 - \sqrt{\frac{V^2 m^3}{4 \sigma^2 n^2}}\big) \geq \frac{V^2 m^2}{48 n} \big(1 - \sqrt{\frac{V^2 m^3}{4 \sigma^2 n^2}}\big) = r_m,
\end{equation*}
where we have used $km \geq n/2$ in the last inequality. 

Note that
the last equation gives a minimax lower bound depending on $m$,
which can be chosen to be any positive integer not bigger than
$n$. Observe, however, that choosing $m = 1$ results in the degenerate
case where all the $\mu^{\tau}$ defined in~\eqref{assouadpacking} are the same vector.
We can therefore write the minimax lower bound in the form
\begin{equation}\label{keyp}
R_{n,V,\sigma} = \inf_{\tilde{\mu}} \sup_{\mu \in \M_{V}} \E \|\tilde{\mu} - \mu\|^2 \geq \max_{2 \leq m \leq n} \frac{V^2 m^2}{48 n} \big(1 - \sqrt{\frac{V^2 m^3}{4 \sigma^2 n^2}}\big).
\end{equation}
We now have three cases to investigate.
\begin{enumerate}
\item \textbf{$n V^2 \leq \min\{n \sigma^2,\, (\sigma^2)^{2/3} (nV^2)^{1/3}\}$}: 
In this case, we have $nV^2 \leq \sigma^2$. In this case set $m = n$ in the right side of~\eqref{keyp} to get the minimax lower bound $$R_{n,V,\sigma} \geq \frac{n V^2}{96}.$$

\item \textbf{$n \sigma^2 \leq \min\{nV^2,\, (\sigma^2)^{2/3} (nV^2)^{1/3}\}$}: In this case we have $n \sigma \leq V$. This implies $R_{n,V,\sigma} \geq R_{n,n \sigma/(2 \sqrt{2}),\sigma}$. Now a lower bound for $R_{n,n \sigma/(2 \sqrt{2}),\sigma}$ can be obtained by setting $m = 2$ in~\eqref{key} to obtain $$R_{n,V,\sigma} \geq \frac{n \sigma^2}{192}.$$

\item \textbf{$(\sigma^2)^{2/3} (nV^2)^{1/3} < \min\{n \sigma^2,\, nV^2\}$}: In this case we have $V < n \sigma$ and $\sigma^2 < n V^2$. We further subdivide this case into two subcases. 

\begin{enumerate} 
\item[a)] Suppose $\sigma^2 n^2/V^2 < 8$. Then we can write
\begin{align*}
R_{n,V,\sigma} \geq R_{n,n \sigma/(2 \sqrt{2}),\sigma} \geq \frac{n \sigma^2}{192} \geq \frac{(\sigma^2)^{2/3} (nV^2)^{1/3}}{192}.
\end{align*}
The first inequality follows from setting $m = 2$ in~\eqref{key} and the second inequality follows because $(\sigma^2)^{2/3} (nV^2)^{1/3} \leq n \sigma^2.$

\item[b)] Suppose $\sigma^2 n^2/V^2 \geq 8$. We now set $m = \lfloor \big(\sigma^2 n^2)/V^2\big)^{1/3}\rfloor$. We then have $m \geq 2$ and also we must have $m \leq n$ because $\sigma^2 < n V^2$. Hence, $m$ is a feasible choice to be used in~\eqref{key}. We now use our choice of $m$ in~\eqref{key} to obtain $$R_{n,V,\sigma} \geq \frac{(\sigma^2)^{2/3} (n V^2)^{1/3}}{192}.$$
Here we have used the fact that $m \geq \big(\sigma^2
n^2)/V^2\big)^{1/3}/2$ and $\frac{V^2 m^3}{4 \sigma^2 n^2} \geq
1/4$. The first fact is true because $m \geq 2$ and the second fact
follows trivially from the definition of $m$.  
\end{enumerate}
\end{enumerate}
This finishes the proof of the theorem.
\end{proof}

\subsection{Proof of Theorem~\ref{minimaxisomain}: Minimax lower bound for deep trees, $\alpha \leq \frac{1}{3}$}
\begin{proof}[Proof of Theorem~\ref{minimaxisomain}]
Denote the subset of flows in $\F(\T_{\alpha,n})$ with root value at
most $V$ by $\F_{n,\alpha,V}$. Recall $m =
\lceil n^{\alpha}\rceil$. Fix any $(v_1,\dots,v_m)$ such that $\sum_{i
  = 1}^{m} v_i \leq V$. Consider the space of flows
$\F(n,V,(v_1,\dots,v_m))$ where the root is set at $V$ and its $m$
children are set at $(v_1,\dots,v_m)$. Clearly
$\F(n,V,(v_1,\dots,v_m)) \subset \F_{n,\alpha,V}$. Hence we have
\begin{equation*}
\inf_{\tilde{\mu}} \sup_{\mu \in \F_{n,\alpha,V}} \E \|\tilde{\mu} - \mu\|^2 \geq \inf_{\tilde{\mu}} \sup_{\mu \in \F(n,V,(v_1,\dots,v_m))} \E \|\tilde{\mu} - \mu\|^2,
\end{equation*}
and it suffices to lower bound the right side of the above
equation. Now the estimation problem in $\F(n,V,(v_1,\dots,v_m))$ is
just $m$ separate isotonic regression problems. This means that the
inf sup term over $\F(n,V,(v_1,\dots,v_m))$ decomposes into a sum of
inf sup terms over each of the paths with monotonicity constraints.
Applying the minimax lower bound for isotonic regression given in
Theorem~\ref{isominimax} to each of these subproblems completes the proof.
\end{proof}

\begin{proof}[Proof of Corollary~\ref{isominimaxcoro}]
Set $v_i = \sigma V/n^{\alpha}$ in the lower bound given by
Theorem~\ref{minimaxisomain}. Since $n^{3 \alpha - 1} \leq V \leq n$,
\begin{equation*}
\min \{n^{1 - \alpha} \sigma^2, (v_i \sigma^2)^{2/3} n^{(1 - \alpha)/3}, n^{1 - \alpha} v_i^2\} = (v_i \sigma^2)^{2/3} n^{(1 - \alpha)/3}.
\end{equation*}
Theorem~\ref{minimaxisomain} then implies that a valid minimax lower
bound in terms of 
the sum of squared errors $n^{\alpha} (v_i \sigma^2)^{2/3} n^{(1 - \alpha)/3} = (V \sigma^2)^{2/3} n^{1/3}$ by the choice of $v_i$.
\end{proof}

\subsection{Proof of Theorem~\ref{simpminimax}: Minimax lower bound for deep trees, $\alpha \geq \frac{1}{3}$}
The proof of Theorem~\ref{simpminimax} is very similar to that of
Theorem~\ref{simplexminimax},
 where we use the Varshamov-Gilbert Lemma~\ref{vg} and Fano's
 lemma~\ref{fano}.  The details are omitted.

\subsection{Proof of Theorem~\ref{minimaxupbd}: Tightness of the
  minimax lower bound, $\alpha \geq \frac{1}{3}$}
The following lemma gives an upper bound to the minimax rate in a general Gaussian denoising problem where the mean is known to lie in a set $K$. The upper bound is information-theoretic and is expressed in terms of covering numbers of $K$. 
\begin{theorem}\label{minimaxupbdp}
Let $y \sim N(\theta^*,\sigma^2 I)$ be a $n$ dimensional random
vector. Let $K \subset \R^n$ and $\theta^* \in K$. Then
\begin{equation*}
\inf_{\tilde{\theta}} \sup_{\theta^* \in K} \E \|\tilde{\theta} - \theta^*\|^2 \leq \inf_{\epsilon > 0} \big(16 \sigma^2 \log N(\epsilon,K,\|.\|) + 3 \epsilon^2\big).
\end{equation*}
\end{theorem}

\begin{proof}
Let $F \subset K$ be a finite subset. Define the least squares estimator over the finite set $F$ as
\begin{equation*}
\hat{\theta}_{F} = \argmin_{\mu \in F} \|y - \mu\|^2.
\end{equation*}
We start with the following inequality which holds for any nonnegative function $G:F \rightarrow \R_{+}$:
\begin{equation*}
G(\hat{\theta}_{F}) \leq \sum_{\mu \in \F} G(\mu) \exp\big(\alpha \|y - \hat{\theta}_{F}\|^2 - \alpha \|y - \mu\|^2\big).
\end{equation*}
Because $\hat{\theta}_{F}$ is the least squares estimator, we can
replace it by any arbitrary but fixed and data-independent 
$\theta' \in F$ in the right side of the above inequality. Then taking expectations on both sides we obtain
\begin{equation}\label{eq1}
\E\:G(\hat{\theta}_{F}) \leq \sum_{\mu \in \F} G(\mu) \E \exp\big(\alpha \|y - \theta'\|^2 - \alpha \|y - \mu\|^2\big).
\end{equation}
Writing $y = \theta^* + z$ where $z \sim N(0,\sigma^2 I)$, some elementary algebra gives us
\begin{equation*}
\|y - \theta'\|^2 - \|y - \mu\|^2 = \|\theta^* - \theta'\|^2 - \|\theta^* - \mu\|^2 + 2 \langle z, \mu - \theta' \rangle.
\end{equation*}
Knowing the moment generating function of $z$ then lets us conclude 
\begin{equation*}
\E \exp\big(\alpha \|y - \theta'\|^2 - \alpha \|y - \mu\|^2\big) = \exp\big(\alpha \|\theta^* - \theta'\|^2 - \alpha \|\theta^* - \mu\|^2 + 2 \alpha^2 \sigma^2 \|\mu - \theta'\|^2\big).
\end{equation*}
The elementary inequality $\|\mu - \theta'\|^2 \leq 2 \|\mu -
\theta^{*}\|^2 + 2 \|\theta^* - \theta'\|^2$ applied to the last 
equation, together with \eqref{eq1}, gives us 
\begin{equation}
\E\:G(\hat{\theta}_{F}) \leq \sum_{\mu \in \F} G(\mu) \exp\big((4\alpha^2 \sigma^2 - \alpha) \|\mu - \theta^{*}\|^2 + (4\alpha^2 \sigma^2 + \alpha) \|\theta^* - \theta'\|^2\big).
\end{equation}
The choices $G(\mu) = \exp(\frac{\|\mu - \theta^{*}\|^2}{16
  \sigma^2})$ and $\alpha = \frac{1}{8 \sigma^2}$ then establish
the risk bound
\begin{equation*}
\E \exp(\frac{\|\hat{\theta}_{F} - \theta^{*}\|^2}{16 \sigma^2}) \leq |F| \exp(\frac{3}{16 \sigma^2} \|\theta^* - \theta'\|^2).
\end{equation*}
Using Jensen's inequality on the left side and taking logarithms yields
\begin{equation*}
\E\:\frac{\|\hat{\theta}_{F} - \theta^{*}\|^2}{16 \sigma^2} \leq \log |F| + \frac{3}{16 \sigma^2} \|\theta^* - \theta'\|^2.
\end{equation*}
Since $\theta' \in F$ was arbitrary we can actually conclude
\begin{equation*}
\E\:\frac{\|\hat{\theta}_{F} - \theta^{*}\|^2}{16 \sigma^2} \leq \log |F| + \min_{\mu \in F} \frac{3}{16 \sigma^2} \|\theta^* - \mu\|^2.
\end{equation*}
Now, if $F = F_{\epsilon}$ is chosen to be a $\epsilon$ cover for $K$,
we have
\begin{equation*}
\E\:\frac{\|\hat{\theta}_{F} - \theta^{*}\|^2}{16 \sigma^2} \leq \log N(\epsilon,K,\|.\|) + \frac{3}{16 \sigma^2} \epsilon^2.
\end{equation*}
Taking the infimum over $\epsilon > 0$ finishes the proof.
\end{proof}

\begin{remark}
This basic idea of the above result 
can be traced back to the paper~\cite{barron2008mdl} and the
references therein. 
A more general version of the above theorem can be found
in~\cite{chatterjee2014adaptation} (Theorem 1.2.2).
\end{remark}

\begin{proof}[Proof of Theorem~\ref{minimaxupbd}]
We use Lemma~\ref{lem:cover} to get an upper bound on the log covering
number of $\F_{n,\alpha,V}$
of the form
\begin{equation*}
\log N(\epsilon,\F_{n,\alpha,V},\|.\|) \leq (1 + \frac{V^2 n^{1 - \alpha}}{\epsilon^2}) (1 + \log(en)).
\end{equation*}
An application of Theorem~\ref{minimaxupbd} immediately yields an
upper bound to the minimax rate $R$ as
\begin{equation*}
\frac{R}{16 \sigma^2} \leq \inf_{\epsilon > 0} \big((1 + \frac{V^2 n^{1 - \alpha}}{\epsilon^2}) (1 + \log(en)) + \frac{3}{16 \sigma^2} \epsilon^2\big).
\end{equation*}
Setting $\epsilon = \big(\frac{16 \sigma^2 V^2 n^{1 - \alpha} \log(en)}{3}\big)^{1/4}$ we obtain
\begin{equation*}
R \leq 12 \big(\frac{16 \sigma^2 n^{1 - \alpha} \log(en)}{3}\big)^{1/2} \leq 21\:V \:\sigma\:n^{(1 - \alpha)/2} \log(en)^{1/2} + 16 \sigma^2 \big(1 + \log(en)\big)
\end{equation*}
as an upper bound to the risk.
\end{proof}

\subsection{Proof of Theorem~\ref{lselbd}: Tightness of the LSE upper
  bound, $\alpha \geq \frac{1}{3}$}
\label{zeroproof}

\begin{proof}
We will again prove the lower bound to the risk at the origin, taking
$\mu = 0$. 
Lemma~\ref{cone} implies the pointwise inequality
\begin{equation*}
\|\hat{\mu} - \mu\| = \sup_{\nu \in \F_{n,\alpha}: \|\nu\| \leq 1} \langle Z,\nu \rangle.
\end{equation*}
Therefore we can now write
\begin{equation}\label{gw}
\big(\E \sup_{\nu \in \F_{n,\alpha}: \|\nu\| \leq 1} \langle Z,\nu \rangle\big)^2 \leq \E \big(\sup_{\nu \in \F_{n,\alpha}: \|\nu\| \leq 1} \langle Z,\nu \rangle\big)^2 = \E \|\hat{\mu} - \mu\|^2.
\end{equation}
Thus, it suffices to lower bound the Gaussian width term $\E \sup_{\nu \in \F_{n,\alpha}: \|\nu\| \leq 1} \langle Z,\nu \rangle$.

Consider the case $1/3 \leq \alpha \leq 1/2$. For $1 \leq i \leq m$ define $S^{(i)} = \sum_{j = 1}^{l} Z^{(i)}_j$ and define the random signs 
\begin{equation*}
s^{(i)} = sign(S^{(i)}).
\end{equation*}
Now define a random flow $v$ in the following fashion. For each $1 \leq i \leq m$ and $1 \leq j \leq l$ define $$v^{(i)}_j = s^{(i)} \frac{1}{2 \sqrt{n}}.$$ Also define $v_{1} = 1/2$. It is easy to check that $v \in \F_{n,\alpha}$ and $\|v\| \leq 1$. Therefore, we can write 
\begin{align*}
\E \sup_{\nu \in \F_{n,\alpha}: \|\nu\| \leq 1} \langle Z,\nu \rangle
\geq \: \E \langle Z,v \rangle & = \E \big(\frac{Z_1}{2} + \frac{1}{2
  \sqrt{n}} \sum_{i = 1}^{m}|S^{(i)}|\big)  \\
& = \frac{1}{2 \sqrt{n}} \sum_{i = 1}^{m} \E |S^{(i)}|\\
& = \frac{m}{2 \sqrt{n}} 2 \pi \sqrt{l} \sigma \\&
= \pi \sigma n^{\alpha/2}.
\end{align*}
This is because $S^{(i)} \sim N(0,l\sigma^2)$ and $\E|Z| = 2 \pi \sigma$ where $Z \sim N(0,\sigma^2)$.
Using~\eqref{gw} then allows us to conclude
\begin{equation}\label{lb1}
\E \|\hat{\mu} - \mu\|^2 \geq \pi^2 \sigma^2 n^{\alpha}.
\end{equation}

Now let us consider the case when $1/2 \leq \alpha \leq 1$. Define a
random flow $v$ as follows. For each $1 \leq i \leq m$
define $$v^{(i)}_j = s^{(i)} \frac{1}{2 \sqrt{n}}\:\I\{i \leq
\sqrt{n}\}.$$ Also define $v_{1} = 1/2$. 
It is again easy to check that $v \in \F_{n,\alpha}$ and $\|v\| \leq 1$. Hence we have
\begin{align*}
\E \sup_{\nu \in \F_{n,\alpha}: \|\nu\| \leq 1} \langle Z,\nu \rangle
\geq \: \E \langle Z,v \rangle 
& = \E \big(\frac{Z_1}{2} + \frac{1}{2 \sqrt{n}} \sum_{i =
  1}^{\sqrt{n}} |S^{(i)}|\big)  \\
&=\frac{1}{2 \sqrt{n}} \sum_{i = 1}^{\sqrt{n}} \E |S^{(i)}| \\
&= \frac{\sqrt{n}}{2 \sqrt{n}} 2 \pi \sqrt{l} \sigma = \pi \sigma n^{(1 - \alpha)/2}.
\end{align*}
Now using~\eqref{gw} lets us conclude that
\begin{equation*}
\E \|\hat{\mu} - \mu\|^2 \geq \pi^2 \sigma^2 n^{(1 - \alpha)},
\end{equation*}
completing the proof of the theorem.
\end{proof}

\section{Simulations}
\label{sec:simulations}

In this section we present results from simulations to gain a
qualitative understanding of the rates of convergence of the least
squares estimator. We investigate the performance of the LSE in the family
of trees $\T_{n,\alpha}$ for various values of $\alpha$. Note that $\alpha = 1$
corresponds to the star graph. For each tree, we repeat the
denoising experiment $100$ times with sample size $n$ growing from
$9{,}000$ to $20{,}000$ in increments of $1000$. In each experiment we
obtain the squared error $\|\hat{\mu} - \mu\|^2$; hence for each sample
size our estimate of the expected sum of squares $\E \|\hat{\mu} -
\mu\|^2$ is an average of $100$ trials. We take the log of
these estimates and fit a linear regression. The slope
is an estimate of the exponent
of increase of $\E \|\hat{\mu} - \mu\|^2$, as
we expect $\E \|\hat{\mu} - \mu\|^2$ to increase like $n^{\nu}$ for some
$\nu$.

We performed simulations for $\alpha = 0.4,0.5,0.75,1$. The true flow
$\mu$ was selected by setting the root vertex
have value $\mu_1=1$. For any $\alpha$, the children of the root were set to
have value $\frac{1}{n^{\alpha}}$ and then the paths decreased from
$\frac{1}{n^{\alpha}}$ to $0$ in equal increments. Such a flow $\mu$ was
chosen because in the case of isotonic regression ($\alpha = 0$),
the mean vector which increases linearly from $0$ to $1$ in increments
of $1/n$ has LSE with error scaling according to
$n^{1/3}$, which is the worst case behavior;
see~\cite{chatterjee2015risk}. Below we show plots (on a log scale) for
the empirical squared error, averaged over $100$ trials, versus
sample size.

\begin{figure}[ht]
\begin{center}
\begin{tabular}{cc}
$\alpha=0.4$ & $\alpha = 0.5$ \\[-10pt]
\includegraphics[width=.4\textwidth]{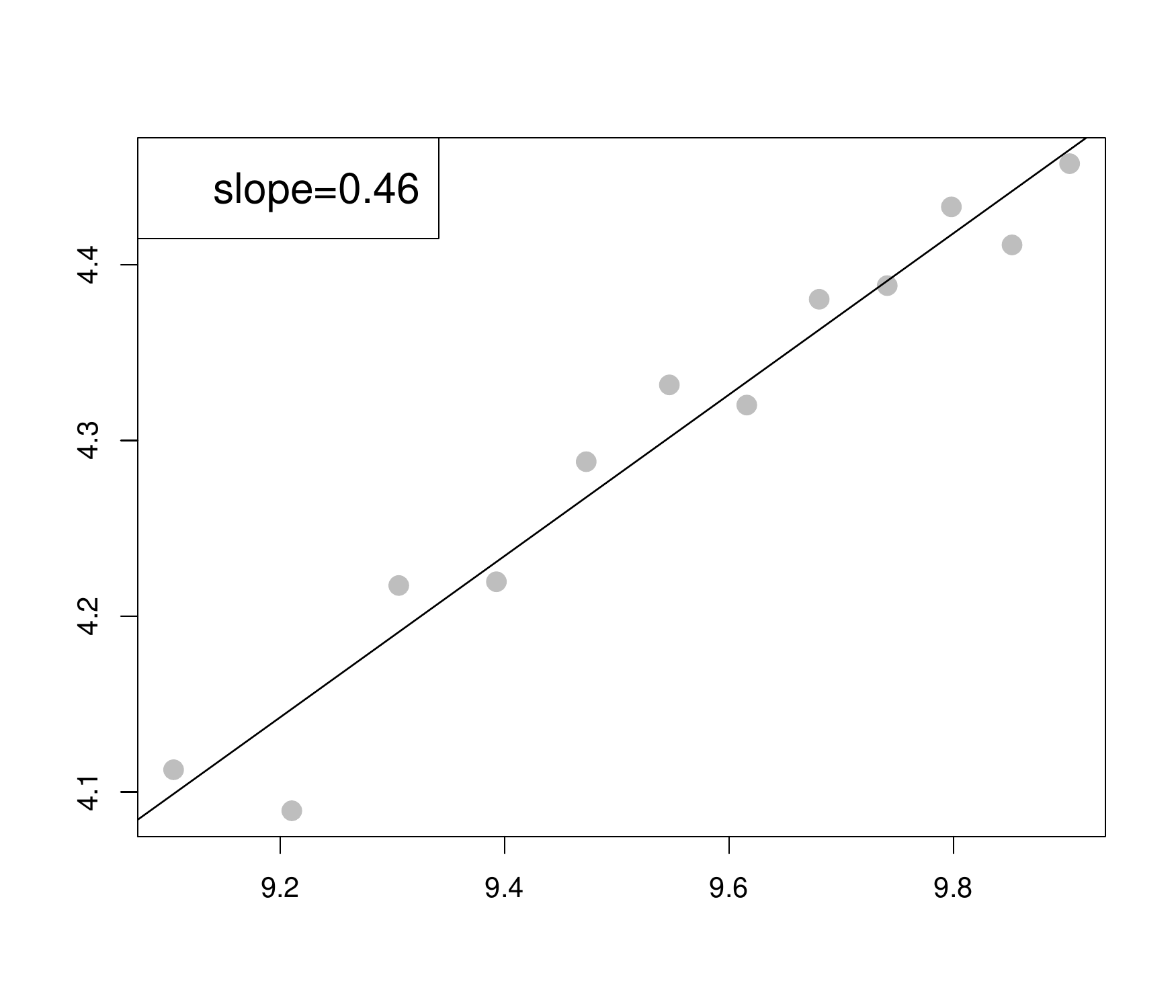} &
\includegraphics[width=.4\textwidth]{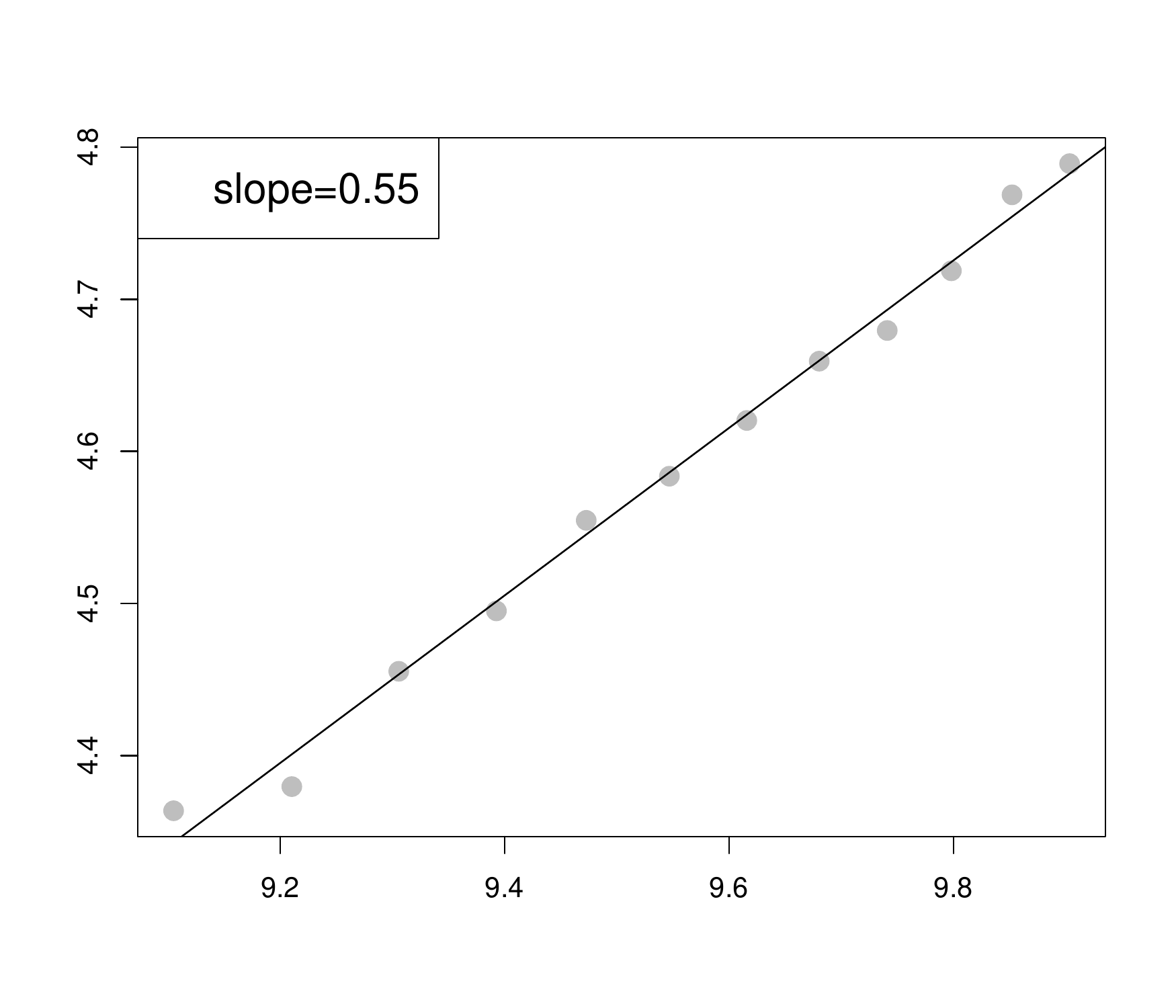} \\
$\alpha=0.75$ & $\alpha = 1$ \\[-10pt]
\includegraphics[width=.4\textwidth]{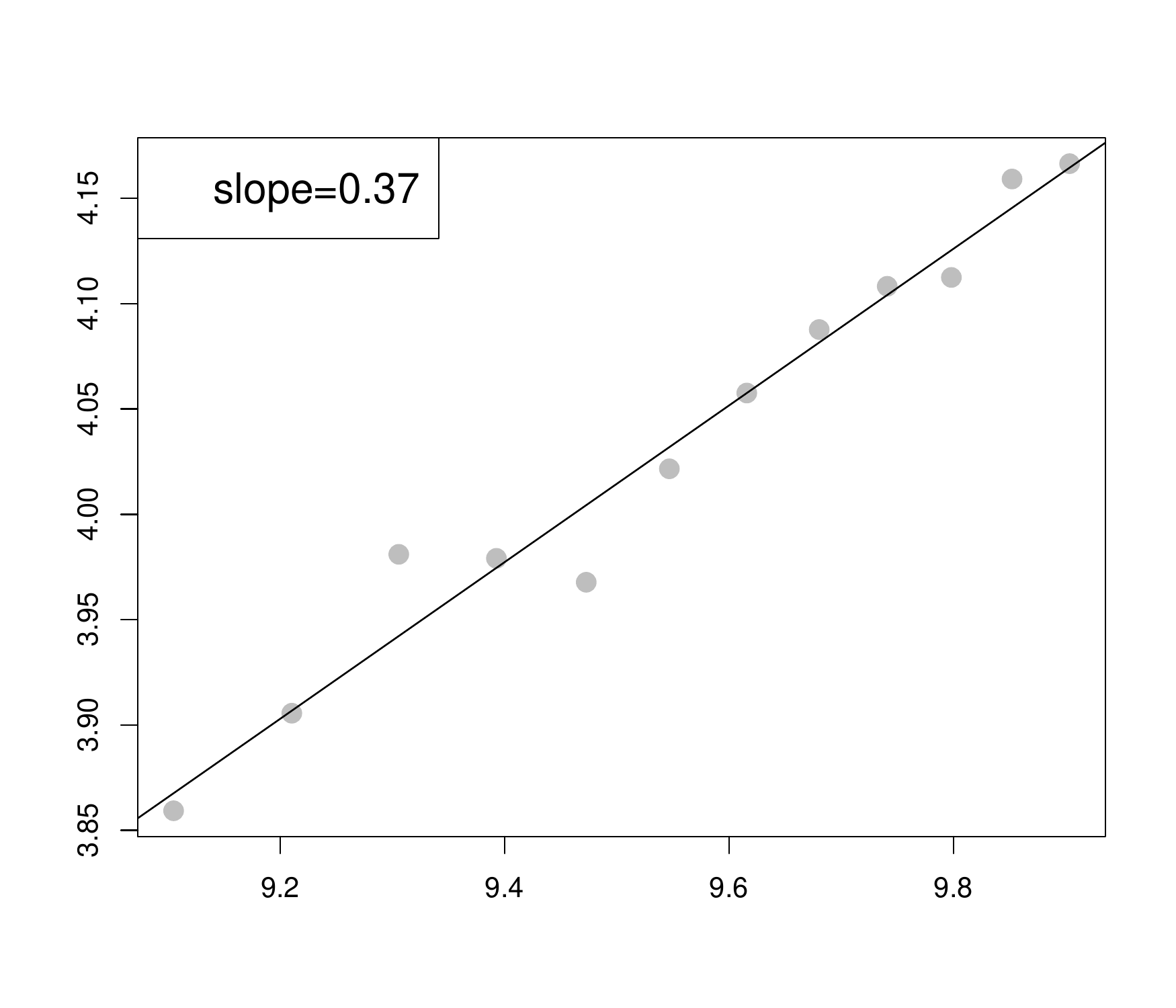} &
\includegraphics[width=.4\textwidth]{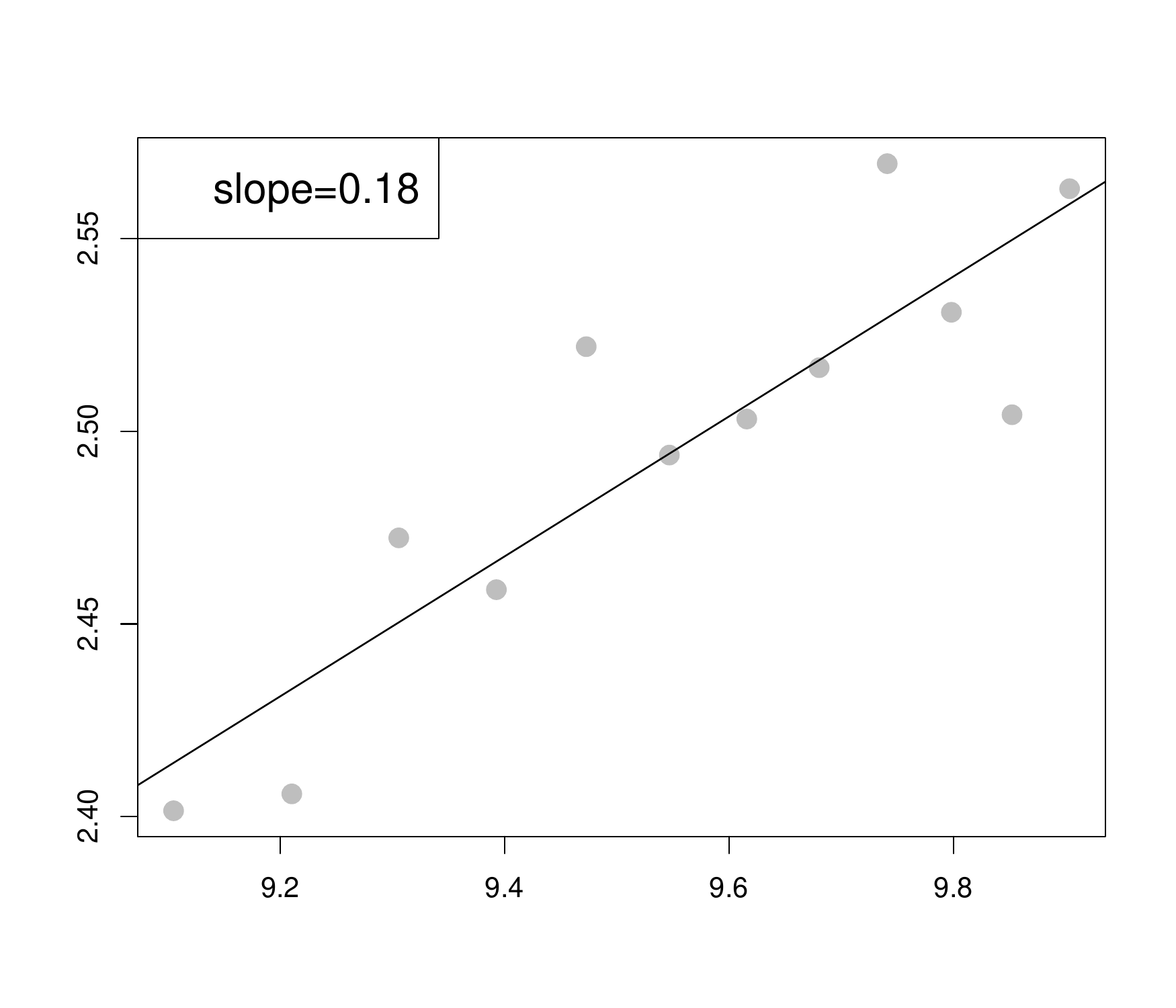} \\
\end{tabular}
\end{center}
\caption{
For each value of $\alpha$, a total of 100 denoising experiments
were run with the size $n$ of the tree growing from
$9,000$ to $20{,}000$ in increments of $1{,}000$. In each experiment we
obtain the squared error $\|\hat{\mu} - \mu\|^2$, and
average over the trials. The slope
is an estimate of the exponent
of increase of $\E \|\hat{\mu} - \mu\|^2$,
as the analysis shows that the expected
squared error increases $n^{\nu}$ for some
$\nu$. The simulations are consistent with the theoretical
findings that the risk bound behaves like $n^{1/2}$ around $\alpha =
0.5$, with the exponent in the bound decreasing linearly for $\alpha$
between $0.5$ and $1$. }
\end{figure}

Note that the upper bound to the expected squared error in our
theoretical analysis gives the exponent $0.4$, $0.5$, $0.25$ for $\alpha =
0.4$, $0.5$ and $0.75$, respectively, and gives logarithmic growth for
$\alpha = 1$. The estimated slopes for $\alpha = 0.4$ and $0.5$ are close to
$0.4$ and $0.5$, while the slopes are slightly higher than our analysis shows
$\alpha = 0.75$ and $1$. This may be because the
logarithmic factors in our risk bound inflates the simulated
slope.

We have simulated only for $\alpha$ larger than $0.4$. This is
because to be able to observe the correct rates in simulations 
for small $\alpha$ would require prohibitively large sample sizes. For
example, with $\alpha=0.1$ and $n < 20{,}000$ note that $n^{0.1} <
3$. Hence, the risk will behave very similarly to 
three separate isotonic regressions---we would observe rates
of convergence scaling like $n^{(1 - 0.1)/3} = n^{0.3}$. In order to
truly reflect the fact that the number of children of the root is
growing like $n^{0.1}$ would require impractically large sample sizes.
We observe, however, that a more computationally tractable estimator
for small~$\alpha$ is $\hat\mu_{\text{nat}}$, which fits the top part of the tree, which is
an $(m+1)$-node star, conditions on the estimated
values of the children, and then fits $m$ separate isotonic regressions.
This algorithms scales linearly.

We have performed the simulations using a generic QP solver from 
Mosek, and the code takes slightly less than a minute for 
one denoising experiment with $n = 20{,}000$. However, the
computation grows quickly for $n$ much larger than $20{,}000$. 
As shown in the plots, the simulations are consistent with our theoretical
findings.

\section{Discussion}

We have formulated a normal means problem on rooted trees, where the
constraints on the means mimic a fluid flowing down from the root with
possible leakage. We have studied the least squares estimator in this
setting, showing risk bounds that only depend on the height of the
tree. For trees with bounded or logarithmically growing diameter, this
gives nearly sharp bounds, which match the minimax lower bounds.  We
have also studied the flow estimation problem for a family of trees
that interpolates between a single path and a star graph.  Here we
find that the rate of convergence of the risk for the LSE is not
monotonic in the depth of the tree. Moreover, we find a gap between
the LSE and the minimax rates. To obtain matching upper and lower
bounds over this family of trees for both the LSE and minimax
estimators, we employ a range proof techniques, as summarized in
Table~1.  Our findings are displayed graphically in Figure~1.

A natural direction for future study is automatic adaptivity of the
LSE to flows that are piecewise constant on long paths.  Such
adaptivity has recently been shown to hold for isotonic regression
\citep{chatterjee2015risk}. Note that we have already obtained a
result in this direction in Theorem~\ref{isotonicupbd}, the proof of
which proceeds by first showing fast rates for flows that are
piecewise constant, and then using an approximation argument.  It
would also be interesting to pursue fast algorithms to denoise flows
on rooted trees.  In particular, one could explore connections to
trend filtering~\citep{wang2014trend}, and the use of penalization and
convex relaxations for flow estimation. For the path graph, the well
known pooled adjacent violators algorithm gives a linear time
procedure \citep{RWD88}. In case the tree is a star graph with root
and $n$ children, the problem of computing the LSE is related to
computing the projection onto a probability simplex in $\R^n$. It may
thus be possible to compute the LSE for certain trees $T_{\alpha, n}$
in $O(n)$ time by slightly modifying the algorithm
in~\cite{duchi2008efficient}.  But our results also present the
challenge of closing the gap between the LSE and the minimax lower
bound with computationally efficient estimators.

\section*{Acknowledgements}
Research supported in part by ONR grant
11896509 and NSF grant DMS-1513594.

\setlength{\bibsep}{8pt}
\bibliographystyle{apalike}
\bibliography{refer}

\end{document}